\documentclass{article}
\title{Multi-type logistic branching processes with selection: frequency process and genealogy for large carrying capacities}
\author{Marta Dai Pra\thanks{Institut f\"ur Mathematik, Humboldt-Universit\"at zu Berlin, Germany, {\tt daimarta@hu-berlin.de}} \ \orcidlink{0009-0001-6464-7722} \and Julian Kern\thanks{Institut f\"ur Mathematik, Freie Universit\"at Berlin, Germany, {\tt j.kern@fu-berlin.de}} \  \orcidlink{0000-0002-8231-0736}}
\date{}
\usepackage[margin=3cm]{geometry}
\usepackage{bbm}
\usepackage{amsmath}
\usepackage{amssymb}
\usepackage{amsfonts}
\usepackage{amsthm}
\usepackage{wasysym}
\usepackage{dsfont}
\usepackage{todonotes}
\usepackage{enumitem}
\usepackage{bm}
\usepackage{soul}
\usepackage{comment}

\usepackage{graphicx}
\usepackage{wrapfig}
\usepackage{tikz}
\usetikzlibrary{shapes.geometric}
\usepackage{orcidlink}
\usepackage{hyperref}
\usepackage{cleveref}
\usepackage[style=trad-alpha,backend=biber, backref=false, url=false, isbn=false, eprint=false]{biblatex}
\definecolor{bluegreen}{rgb}{0.64, 0.76, 0.68}

\newcommand{\abs}[1]{\lvert#1\rvert}

\renewcommand*{\P}{\mathbb{P}}
\newcommand*{\E}{\mathbb{E}}
\newcommand*{\R}{\mathbb{R}}

\newcommand*{\N}{\mathbb{N}}
\renewcommand*{\Xi}{\varXi}
\renewcommand*{\epsilon}{\varepsilon}
\renewcommand*{\Theta}{\varTheta}

\renewcommand*{\Delta}{\varDelta}

\newcommand*{\1}{\mathds{1}}
\newcommand{\dx}{\;\mathrm{d}}

\newcommand{\s}{\mathfrak{s}}

\usepackage{textcase}
\usepackage{xcolor}

\theoremstyle{plain}
\theoremstyle{definition}
\newtheorem{definition}{Definition}[section]

\newtheorem{lemma}[definition]{Lemma}

\newtheorem{theorem_internal}[definition]{Theorem}
\theoremstyle{plain}
\newtheorem{theorem}{Theorem}
\theoremstyle{remark}
\newtheorem{remark}[definition]{Remark}

\newtheorem{assumption}{Assumption}

\begin{document}
\maketitle
\begin{abstract}
	We present a model for growth in a multi-species population. 
    We consider two types evolving as a logistic branching process with mutation, where one of the types has a selective advantage, and are interested in the regime in which the carrying capacity of the system goes to $\infty$. 
    We first study the genealogy of the population up until it almost reaches carrying capacity through a coupling with an independent branching process. 
    We then focus on the phase in which the population has reached carrying capacity. After recovering a Gillespie--Wright--Fisher SDE in the infinite carrying capacity limit, we construct the Ancestral Selection Graph and show the convergence of the lineage counting process to the moment dual of the limiting diffusion.
    \medbreak
\end{abstract}

\section{Introduction}
Among the many processes used in the study of population evolution, two of the most important ones are the Wright--Fisher diffusion and supercritical branching processes. 
In this work we consider a model that links them together to create a more complete description of a growing population.

Supercritical branching processes are characterised by exponential growth and are used to describe growing populations starting with few individuals. 
The exponential growth may be realistic as long as the population is small; when the population is large, however, it is biologically more realistic that its growth slows down due, for example, to competition for a limited resources. 
The Wright--Fisher diffusion, on the other hand, usually arises as scaling limit of neutral populations when their constant (in time) size grows to infinity. 
The gap between these two scenarios of exponential growth and large populations evolving as a diffusion can be bridged through density-dependent branching processes. 

These processes, where individuals reproduce and die with rates depending on the population size, have been widely studied in the past \cite{Lip75, Kleb84, Kle93}, and have regained attention in recent years \cite{Che22,FRW24,For25}. The model we are particularly interested in is the logistic branching process introduced in \cite{Camp03,Lambert05}, where individuals reproduce and die at constant rate but can also die because of competition with other individuals at a rate that grows linearly with the population size. This model shows an initial phase of exponential growth but, eventually, the population fluctuates around a value called \emph{carrying capacity}. In this paper we are interested in the limiting behaviour of such a process as the carrying capacity $K$ grows to infinity. 

In all the models cited so far, individuals reproduce in the same way, i.e.\ their random number of offspring is identically distributed. 
A natural generalisation is to let different individuals have different reproduction mechanisms. 
Here, we study a population whose individuals carry one of two possible types $\{+,-\}$ characterised by two different reproduction mechanisms. Individuals of type $+$ can be thought as selectively advantaged with respect to the type $-$.

In this work, we will investigate three convergence results. 
The first one focuses on the genealogy of the process in the initial growth phase until it almost reaches carrying capacity. 
More precisely, we will extend a corresponding result on one-type branching processes by \cite{Che22}, for which they introduced a coupling argument to show that the genealogy of the logistic branching process is well-approximated by the corresponding independent branching process as long as the growth rate is superlinear, which is true until the process almost reaches carrying capacity. 
We show here that these results still hold in our setting of multi-type logistic branching process with selection.

The second and third main results address the phase in which the population size is comparable to $K$, by looking at the process forward and backward in time, respectively.
In a finite variance regime, we obtain the limiting behaviour of the two-type frequency process: the wild-type frequency converges, as $K \to \infty$, to the unique strong solution of an SDE of the form
\begin{equation*}
    dX(t)=aX(t)(1-X(t))dt+\sqrt{bX(t)(1-X(t))+cX^2(t)(1-X(t))}dB(t),
\end{equation*}
for some constants $a\in \R$ and $b,\,c \geq 0$. 
The above diffusion, which presents a Wright--Fisher component, was introduced by Gillespie in \cite{G75}, and is known to describe the limiting frequency for a multi-type Wright--Fisher model with type-dependent reproduction mechanism \cite{G75,S78}. 
The same kind of equation has recently reappeared independently in the context of so-called Wright--Fisher models with efficiency, see e.g.\ the discussion in \cite{CPP21}. 
Our convergence proof relies heavily on the techniques introduced by \cite{Kat91} to study solutions of SDEs that are pushed onto a manifold by a large drift, see \Cref{sec:Katzenberger-method} for details.
This allows us to simplify the proofs based on the more classical approach of first investigating the convergence of the total population size and then considering the frequency process separately, see e.g.~\cite{KernWied24,For25}.

Under rather weak assumption on the expectation and variance of the offspring distribution, our limiting frequency process has a well-defined moment dual. 
As in many population genetics models (e.g.~Wright--Fisher model and Kingman coalescent), we would like to interpret the dual process as the lineage counting process of the (potential) genealogy of a sample of individuals. 
In the neutral case with the same finite-variance-type reproduction mechanism for both types, this has been done in the seminal work of Donnelly and Kurtz, \cite[Theorem 5.1]{DK99}. 
As soon as different reproduction mechanisms compete, their methodology breaks down, as it relies on the fact that the type is not relevant to the forward-in-time picture. 
In the recent work \cite{CKP23}, the authors introduce a $\Lambda$-Moran model with selection based on the assumption that the reproduction mechanisms are stochastically ordered. 
This assumption (see \eqref{eq:stoch-ordering}) has proved very helpful also in our setting where the population is not constant in size. 
In this scenario, we are indeed able to define the Ancestral Selection Graph (ASG) of the population and show that the corresponding lineage counting process converges to the dual of the limiting frequency process.\\

Finally, we would like to highlight two open questions. 
Firstly, the convergence of genealogies before reaching carrying capacity is proved herein and in \cite{Che22} for an arbitrary but finite sample size of individuals.
In the setting of the Wright--Fisher model with population size $K$, it is known that the genealogy undergoes phase transitions when the sample size grows with $K$, see e.g.~\cite{phase_transition_kingman} and the references therein.
As such, it would be of great interest to study in more details how the genealogy of the branching process behaves in a similar context, particularly before reaching carrying capacity.

The second open problem concerns the interpretation of the duality relation at carrying capacity.
Similarly to \cite{CKP23}, the interpretation of the dual process as a lineage counting process can only be made precise when we can couple the reproduction rates of the different types.
This stems from the fact that such a coupling yields a graphical construction of the process where the dependence on the types can be mostly suppressed.
We observe, however, that there is an intermediate regime for which we can obtain the moment duality but for which such a coupling is impossible.
This includes in particular the most intriguing regime in which the two reproductive mechanisms are allowed to have different variances, cf.~\cite{G75}.
It remains open whether it is possible to extend the interpretation of the moment duality as a sampling duality also in this regime.\\

The paper is structured as follows: we start by introducing the model and stating the three main results in \Cref{sec:model}; in \Cref{sec:initial-phase}, we prove the convergence of the genealogy in the first phase of superlinear growth; we then present, in \Cref{sec:Katzenberger-method}, the result of \cite{Kat91} and adapt it to our setting, so that we can use it to find the limit of the frequency process at carrying capacity in \Cref{sec:limit-carrying-capacity}; finally, in \Cref{sec:limit-genealogy} we define the ASG of the population at carrying capacity and prove the convergence of the corresponding lineage counting process.

\section{Model and main results}\label{sec:model}

Let us consider a population in which individuals carry one of two possible types in $\{+,-\}$. Each individual reproduces independently at rates determined by two finite measures $\mu_K^+$ and $\mu_K^-$ on $\N_0$. 
In particular, at rate $\mu_K^\pm(i)$, an individual of type $\pm$ produces $i$ children and then dies. 
The parameter $K$ represents the carrying capacity of the population and we will be interested in the regime $K\to \infty$. 
We assume that there exist constants $\mathfrak{m}>0$ and $\s^+,\,\s^-\in\R$ such that the mean reproduction rates $\mathfrak{m}^{\pm}_K=\sum_{i\geq 0} (i-1) \mu_K^\pm(i)$ satisfy
\begin{equation}\label{eq:marta_ass_on_means}
\mathfrak{m}_K^+ = \mathfrak{m} + \dfrac{\s^+}{K}+\text{o}\left(\frac{1}{K}\right)\qquad\text{ and }\qquad \mathfrak{m}_K^- = \mathfrak{m} + \dfrac{\s^-}{K} + \text{o}\left(\dfrac{1}{K}\right).
\end{equation}
Later on, we will assume $\mathfrak{s}^+ \geq \mathfrak{s}^-$, but for the moment, we do not assume them to be ordered.

To ensure that the population fluctuates around the large carrying capacity $K$, individuals do not only die at the natural rate $\mu_K^\pm(0)$, but have an additional type-independent death rate of $\mathfrak{m}/K$ times the current population size induced by competition with other individuals.
Finally, we assume that individuals can change their type with mutation rates $\theta^{+}/K$ from type $+$ to type $-$ and $\theta^-/K$ from $-$ to $+$.

The size of the above-described population evolves as a multi-type logistic branching process; more precisely, it is defined by a continuous-time Markov process $\left(N_K^+(t),N_K^-(t)\right)_{t \geq 0}$ on $\N_0 \times \N_0$ with transitions

 \begin{equation}\label{eq:pop-size-rates}
 \arraycolsep=3pt\def\arraystretch{1.9} (n,m) \mapsto   \left\{ \begin{array}{llllll}
          (n+i-1,m) & \text{at rate} \hspace{0.2cm}& n\,\mu^+_K(i) \quad \text{for } i\geq 1, \\
       (n,m+i-1) & \text{at rate} \hspace{0.2cm}& m\,\mu^-_K(i) \quad \text{for } i\geq1,\\
       (n-1,m) & \text{at rate} \hspace{0.2cm}& n\,\mu^+_K(0)+n(n+m)\frac{\mathfrak{m}}{K},\\
       (n,m-1) & \text{at rate} \hspace{0.2cm}& n\,\mu^-_K(0)+m(n+m)\frac{\mathfrak{m}}{K},\\
       (n-1,m+1) & \text{at rate} \hspace{0.2cm}& n\,\frac{\theta^+}{K},\\
       (n+1,m-1) & \text{at rate} \hspace{0.2cm}& m\,\frac{\theta^-}{K}.
            \end{array} \right. 
        \end{equation}
We denote by $(N_K(t))_{t \geq 0}$ the total population size, so that $N_K(t)=N^+_K(t)+N_K^-(t)$ for all $t\geq 0$. 
We put ourself in a finite variance regime and assume the existence of two limiting finite measures $\mu_\infty^\pm$ as well as of two positive numbers $\mathfrak{v}^+,\,\mathfrak{v}^->0$ such that
\begin{equation*}
    \sum_{i\geq 0} (i-1)\mu^\pm_{\infty}(i)=\mathfrak{m} \quad \quad \lim_{K\to \infty}\sum_{i\geq 0} (i-1)^2 \mu_K^\pm(i)=\sum_{i\geq 0} (i-1)^2\mu^\pm_{\infty}(i)=\mathfrak{v}^\pm,
\end{equation*}
and 
\begin{equation}\label{eq:convergence-measure}
\lim_{K \to \infty} \sum_{i\geq 0}\,\abs{\mu^{\pm}_K(i)-\mu^\pm_{\infty}(i)}=0.
\end{equation}
The definition of the branching Markov process can be naturally extended to $K = \infty$ by setting $\frac{1}{\infty} := 0$: the processes $(N_\infty^+(t))_{t \geq 0}$ and $(N_\infty^-(t))_{t \geq 0}$ become independent branching processes with combined transitions 
 \begin{equation*}
 \arraycolsep=3pt\def\arraystretch{1.9} (n,m) \mapsto   \left\{ \begin{array}{lll}
          (n+i-1,m) & \text{at rate} \hspace{0.2cm}& n\,\mu^+_\infty(i) \quad \text{for } i\geq 0, \\
       (n,m+i-1) & \text{at rate} \hspace{0.2cm}& m\,\mu^-_\infty(i) \quad \text{for } i\geq0.
            \end{array} \right. 
        \end{equation*}
Note that we do explicitly include the case $\mathfrak{v}^+\neq \mathfrak{v}^-$ which allows for drastically different reproduction mechanisms. 
Similar assumptions have been introduced from a biological perspective in \cite{G75} and a rigorous treatment of the corresponding model, in the constant population setting, can been found in \cite{S78}.

In order to keep track of the complete genealogy of the population we introduce the state space
\[
\mathcal{T}=\cup_{m\in \N}\N^m
\]
of finite tuples. We interpret any $u=(u_1,\dots,u_m)\in \mathcal{T}$ as an individual whose parent is given by the tuple $(u_1,\dots,u_{m-1})$, so that this representation encodes the entire ancestry of $u$. 
This induces the partial ordering $\prec$ of being ancestral to: more precisely two individuals $u=(u_1,\dots,u_m),v=(v_1,\dots,v_n)\in \mathcal{T}$ are ordered, $u\prec v$, if $m<n$ and $u_l=v_l$ for $l=1,\dots,m$. 
We denote the possible offspring of $u=(u_1,\dots,u_m)$ by $ui:=(u_1,\dots,u_m,i)$ for $i \in \N$.
Furthermore, we write
\begin{equation*}
\mathfrak{T} := \{ A \subseteq \mathcal{T}\;:\; \vert A\vert < +\infty\}
\end{equation*}
for the set of finite populations. For ease of notation, we identify $\mathfrak{T}\times \mathfrak{T}$ with the set of finite point processes on $\mathfrak{T}\times \{+,-\}$ through
\begin{equation*}
    (\mathcal{U},\mathcal{V}) \longleftrightarrow \sum_{u \in \mathcal{U}}\delta_u^+ + \sum_{v \in \mathcal{V}}\delta_v^-,
\end{equation*}
where $\delta_u^\pm$ is the Dirac measure on $(u,\pm)$.
The population corresponding to the above logistic branching process can be described through the continuous time Markov chain $(\mathcal{N}_K(t))_{t \geq 0}$ with values in $\mathfrak{T}\times \mathfrak{T}$ and transitions from $\mathcal{N}_K(t)=(\mathcal{N}_K^+(t), \mathcal{N}_K^-(t))$ to
 \begin{equation*}
 \arraycolsep=3pt\def\arraystretch{1.9} \left\{ \begin{array}{lll}
          \mathcal{N}_K(t)-\delta_u^
          \pm+ \sum_{j=1}^i \delta^\pm_{uj} & \text{at rate} \hspace{0.2cm}& \mu^\pm_K(i) \quad \text{for } u\in \mathcal{N}^\pm_K(t),\, i\geq 1, \\
       \mathcal{N}_K(t)-\delta_u^
          \pm & \text{at rate} \hspace{0.2cm}& \mu^+_K(0)+N_K(t) \frac{\mathfrak{m}}{K} \quad \text{for } u\in \mathcal{N}^\pm_K(t),\\
       \mathcal{N}_K(t)\mp \delta_u^
          \pm \pm  \delta_u^
          \mp  & \text{at rate} \hspace{0.2cm}& \frac{\theta^\pm}{K}\quad \text{for } u\in \mathcal{N}^\pm_K(t).\\
            \end{array} \right. 
        \end{equation*}

Following \cite{Che22}, we will introduce a coupling argument to show that the genealogy of the logistic branching process is well approximated by the correspondent independent branching process as long as the growth rate is superlinear. To do so, we will first consider the process up to a stopping time $T_K \to \infty$ 
for which there exist $\alpha\in(0,1)$ and $c>0$ such that 
\begin{equation*}
\lim_{K\to \infty}\P \left( \forall t\in[0,T_K),\quad \mathfrak{m}^{K,N_K(t)}\geq \frac{c}{N_K(t)^\alpha} \right) = 1,	
\end{equation*}
where $\mathfrak{m}^{K,N_K(t)}=(\mathfrak{m}^+_K \wedge \mathfrak{m}_K^-)-N_K(t)\frac{\mathfrak{m}}{K}$ is the \emph{per capita} minimum growth rate. 
Under assumptions \eqref{eq:marta_ass_on_means}, the mean \emph{per capita} growth rate is bounded from below by
\[
\mathfrak{m}\left(1 - \frac{N_K(t)}{K}\right) + \text{o}(1/K),
\]
so that for $N_K(t) \leq K$, we are interested in solving the equation
\begin{equation}\label{eq:marta_superlinear_population}
N_K(t)^\alpha - \dfrac{1}{K}N_K(t)^{1+\alpha} \geq c.
\end{equation}
for some $c>0$. Note that for $N_K(t)=1$ and $K$ large enough, the above equation is satisfied with $c=1/2$, and that the function $x\mapsto x^\alpha - x^{1+\alpha}/K$ is increasing up until $x = \frac{\alpha}{1+\alpha}K$, then decreasing.
In particular, \eqref{eq:marta_superlinear_population} is satisfied as long as $N_K(t)$ is bounded from above by some $n_{max}$ for which there exists $c>0$ such that $1 - n_{max}/K \geq c n_{max}^{-\alpha}$.

For $\beta\in (0,1)$, the above is satisfied before the stopping time $T_K^\beta := \inf\{t\geq 0\;:\; N_K(t)\geq K - K^\beta\}$ with $\alpha = 1-\beta$ as one has indeed
\begin{equation*}
    1 - \dfrac{K - K^\beta}{K} = K^{\beta - 1}.   
\end{equation*} 
This particular choice has the added convenience that the evolution of the process after the stopping time $T_K^\beta$ effectively starts from carrying capacity.

Before stating the first main result, we define the following quantities. For each individual $u \in \mathcal{T}$ we define by
\begin{equation*}
D_K^u(t)=\#\{v\in \mathcal{N}_K(t):u \preceq v \},
\end{equation*}
the number of descendants of $u$ that are alive at time $t \geq 0$, and by
\begin{equation*}
F^u_K(t)=\1_{\{N_K(t)>0\}}\frac{D_K^u(t)}{N_K(t)},
\end{equation*}
the fraction of the population descending from $u$.

We write $F_\infty(\infty) := \lim_{t\to +\infty} F_\infty(t)$ for the asymptotic fraction of descendants of $g$ in the limiting branching process without interactions; the existence of which follows as in \cite[Lemma 4.2]{Che22}.
The following result concerning the growth phase before hitting carrying capacity is in the spirit of \cite[Proposition 5.4]{Che22}. 
Using the arguments from \cite[Section 7]{Che22}, this also proves the convergence of the corresponding genealogies.

\begin{theorem}{\emph{(Initial growth)}}\label{thm:application_of_comparison}
Consider the stopping time $T^\beta_K := \inf\{t\geq 0\;:\; N_K(t) \geq K - K^\beta\}$ for some $\beta \in (0,1)$.
Then there exists a coupling such that,
\[
\lim_{K\to\infty} F_K^u(T^\beta_K) = F_\infty^u(\infty),
\]
in probability, for every $u\in\mathfrak{T}$.
\end{theorem}

On the event $\{T^\beta_K<\infty\}$, we can investigate the behaviour of the population around carrying capacity. 
For simplicity, we will assume from now on that the logistic branching process starts at carrying capacity, i.e.\ $N_K(0)/K\Rightarrow 1$ as $K\to \infty$, where $\Rightarrow$ denotes the convergence in distribution. 
As the rescaled population size concentrates around $1$, we may consider, instead of the frequency process of the two types when time is rescaled by a factor $K$, the Markov process $\left(X_K(t)\right)_{t \geq 0}=\left((X_K^+(t),X_K^-(t))\right)_{t \geq 0}$ on $[0,+\infty)^2$ given by
\begin{equation*}
    X_K^\pm:=\frac{N_K^\pm(Kt)}{K}.
\end{equation*}
Our second result regards the convergence of the above frequency process as $K\to \infty$. 

\begin{theorem}{\emph{(At carrying capacity, forward in time)}}\label{theo:convergence_at_carrying_capacity}
    Assume that $X_K(0) \Rightarrow X(0) = (1-W_0, W_0)$ with $W_0$ a.s.\ in $[0,1]$. Assume also that there exists $\eta>0$ such that $\sum_{i\geq 0} i^{2+\eta} \mu_K^\pm(i)$ is uniformly bounded in $K$. 
    Then, $X_K$ converges weakly to $X = (1-W, W)$ where $0\leq W\leq 1$ is the unique strong solution to the Gillespie--Wright--Fisher SDE with mutation
    \begin{equation}\label{eq:Gillespie-WF}
    \begin{split}
        \mathrm{d}W(t) &= -\Big(\big(\mathfrak{s}^+ -\mathfrak{v}^+\big)-\big(\mathfrak{s}^- - \mathfrak{v}^-\big)\Big)W(t)\big(1 - W(t)\big) \mathrm{d}t\\
        &\qquad+ \Big(\theta^+\big(1 - W(t)\big) - \theta^- W(t)\Big)\mathrm{d}t \\
        &\qquad + \sqrt{(\mathfrak{m}+\mathfrak{v}^-)W(t)\left(1-W(t)\right) - (\mathfrak{v}^+ - \mathfrak{v}^-)W^2(t)\left(1-W(t)\right)}\mathrm{d}B(t),
    \end{split}
    \end{equation}
    started from $W(0) = W_0$ and where $B$ is a standard Brownian motion. 
    Furthermore, if 
    \[
    \mathfrak{s}^+ -\mathfrak{v}^+\geq \mathfrak{s}^- -\mathfrak{v}^-,  \quad \mathfrak{v}^+ \geq \mathfrak{v}^-\quad\text{ and }\quad\theta^- = 0,
    \]
    the process $W$ satisfies the moment duality 
    \begin{equation}\label{eq:moment-duality}
      \E\left[\left(W(t)\right)^n \vert W(0)=w\right]=\E\left[w^{A(t)} \vert A(0)=n\right],
    \end{equation}
    for all $t \geq 0$ where $\left(A(t)\right)_{t \geq 0}$ is the $\N_0$-valued Markov branching--coalescing process with annihilation, with generator
    \begin{equation}\label{eq:dual-generator}
    \begin{split}
        \mathcal{L}_{\text{\emph{dual}}}f(n)=&\,\theta^+ n \big(f(n-1)-f(n)\big)\\
        &+ \Big(\big(\mathfrak{s}^+ -\mathfrak{v}^+\big)-\big(\mathfrak{s}^- - \mathfrak{v}^-\big)\Big)n \big(f(n+1)-f(n)\big)\\
        &+ (\mathfrak{m}+\mathfrak{v}^-)\binom{n}{2}\big(f(n-1)-f(n)\big)\\
        &+(\mathfrak{v}^+ - \mathfrak{v}^-)\binom{n}{2}\big(f(n+1)-f(n)\big).
        \end{split}
    \end{equation}
\end{theorem}
As mentioned in the introduction, we would like to interpret the dual process as the lineage counting process of the (potential) genealogy of a sample of individuals from the population. 
In this work, we restrict ourselves to stochastically ordered measures $\mu_K^\pm$, in the sense that 
\begin{equation}\label{eq:stoch-ordering}
  \mu_K^-([i,+\infty))\leq \mu_K^+ ([i,+\infty)),
\end{equation}
for all $i\geq 1$ and 
\begin{equation}\label{eq:intro-ordering_death_rates}
    \mu^+_K(0) \leq \mu^-_K(0),
\end{equation}
see also \cite{CKP23}. Furthermore we assume that the higher moments of the reproduction measures $\mu_K^\pm$ satisfy
\begin{equation}\label{eq:higher-moments}
\sum_{i\geq 0} i^\ell\mu^\pm_{K}(i)= \text{o}(K^{\ell-2}). 
\end{equation}
for all $i\geq 1$ and $\ell\geq 3$. We will show, in \Cref{lem:properties_mixed_moments}, that these assumptions imply $\mathfrak{v}^+ = \mathfrak{v}^-:=\mathfrak{v}$. In this setting, we must have $\mathfrak{s}^+\geq \mathfrak{s}^-$, so that individuals of type $+$ have a selective advantage consisting of a higher mean reproduction rate.
Note that the usual Moran model with selection satisfies these assumptions: indeed in that case we have $\mu_K^-(2)=1$, $\mu_K^+(2)=1+s/K$ for some constant $s\geq 0$ and $\mu_K^\pm(i)=0$ for all $i\neq 2$. 

Now, let $T>0$ be a time horizon and consider the ASG of a sample of $k\geq 2$ individuals at time $KT$ (see \Cref{sec:limit-genealogy} for precise definitions). Let $(A_K(t))_{t \in [0,T]}$ be the corresponding time-rescaled lineage counting process. Our third result states that $A_K$ converges to the dual process \eqref{eq:dual-generator} under the above additional assumption. 

\begin{theorem}{\emph{(At carrying capacity, backward in time)}}\label{thm:block-counting-convergence}
Let $(A_K(t))_{t \in [0,T]}$ be the lineage counting process introduced above and let $\left(A(t)\right)_{t \in [0,T]}$ be the continuous time Markov process with transitions
\begin{equation}\label{eq:lineage-counting-lim}
 \arraycolsep=3pt\def\arraystretch{1.9} n \mapsto   \left\{ \begin{array}{lll}
          n-1 & \text{at rate }& \binom{n}{2}(\mathfrak{m} + \mathfrak{v})+n\,\theta^+, \\
       n+1 & \text{at rate }& n\big(\mathfrak{s^+}-\mathfrak{s^-}\big).
            \end{array} \right. 
        \end{equation}
Assume that conditions \eqref{eq:stoch-ordering}, \eqref{eq:intro-ordering_death_rates} and \eqref{eq:higher-moments} are satisfied as well as all the assumptions of \Cref{theo:convergence_at_carrying_capacity}. Then, if $A_K(0)\Rightarrow A(0)$, the sequence $A_K
$ converges weakly to $A$, as $K \to \infty$, in the space of c\`adl\`ag functions $\mathcal{D}\left([0,T],\N\right)$.
\end{theorem}

\section{The initial growth phase}\label{sec:initial-phase}
In this section, we study the initial growth phase of the branching process until the population ``almost'' reaches carrying capacity, and prove \Cref{thm:application_of_comparison}. Our aim is to show that the genealogy of the population $\mathcal{N}_K$ is well approximated, as the carrying capacity grows to infinity, by the one of the branching process with exponential growth $\mathcal{N}_\infty$. 
The initial growth is characterised by two behaviours: at first, the logistic control has no effect and the population grows exponentially, in the second part, the population starts to behave diffusively, but it does not last long enough to have an effect on the genealogy. We study these in \Cref{cor:cheeky_zeta_K_min_T_K} on the one hand and \Cref{lem:Lemma5.1,lem:Lemma 5.2} on the other. 

\begin{lemma}\label{cor:cheeky_zeta_K_min_T_K}
There exists a coupling and a deterministic sequence of times $\zeta_K\to +\infty$ such that for any $\beta\in (0,1)$ $u\in \mathfrak{T}$,
\begin{equation*}
\lim_{K\to \infty} F^u_K(\zeta_K\wedge T^{\beta}_K)=F^u_\infty(\infty),
\end{equation*}
in probability, where $T_K^\beta := \inf\{t\geq 0\;:\; N_K(t)\geq K - K^\beta\}$.
\end{lemma}
\begin{proof}
The proof, as in \cite[Lemma 4.3]{Che22}, goes through \emph{mutatis mutandis} by considering a coupling of $(\mathcal{N}_K)_{K\in \N \cup \infty}$ in such a way that the $\mathcal{N}_K$, $K\in \N$, are conditionally independent given $\mathcal{N}_\infty$, and with transitions from $(\mathcal{N}_\infty(t),\mathcal{N}_K(t))$ to 

 \begin{equation*}
 \arraycolsep=3pt\def\arraystretch{1.9} \left\{ \begin{array}{llllllll}
          \Big(\mathcal{N}_\infty(t)-\delta_u^
          \pm+ \sum_{j=1}^i \delta^\pm_{uj}\,,\, \mathcal{N}_K(t)-\delta_u^
          \pm+ \sum_{j=1}^i \delta^\pm_{uj}\Big) & \text{at rate } & \mu^\pm_\infty(i) \wedge \mu^\pm_K(i) \quad\\ & & \text{for } u\in \mathcal{N}^\pm_\infty(t) \cap \mathcal{N}^\pm_K(t),\, i\geq 0, \\
          
          \Big(\mathcal{N}_\infty(t)-\delta_u^
          \pm+ \sum_{j=1}^i \delta^\pm_{uj}\,,\, \mathcal{N}_K(t)\Big) & \text{at rate } &  \mu^\pm_\infty(i)-\mu^\pm_\infty(i) \wedge \mu^\pm_K(i)\1_{\{u\in \mathcal{N}_K^\pm(t)\}} \quad\\ & & \text{for } u\in \mathcal{N}^\pm_\infty(t) ,\, i\geq 0, \\
          
          \Big(\mathcal{N}_\infty(t)\,,\, \mathcal{N}_K(t)-\delta_u^
          \pm+ \sum_{j=1}^i \delta^\pm_{uj}\Big) & \text{at rate } &  \mu^\pm_K(i)-\mu^\pm_\infty(i) \wedge \mu^\pm_K(i)\1_{\{u\in \mathcal{N}_\infty^\pm(t)\}} \quad\\ & & \text{for } u\in \mathcal{N}^\pm_K(t) ,\, i\geq 0, \\
          
       \Big(\mathcal{N}_\infty(t)\,,\,\mathcal{N}_K(t)-\delta_u^
          \pm\Big) & \text{at rate} & N_K(t) \frac{\mathfrak{m}}{K} \quad \text{for } u\in \mathcal{N}^\pm_K(t),\\
       \Big(\mathcal{N}_\infty(t)\,,\,\mathcal{N}_K(t)\mp \delta_u^
          \pm \pm  \delta_u^
          \mp\Big)  & \text{at rate} & \frac{\theta^\pm}{K}\quad \text{for } u\in \mathcal{N}^\pm_K(t).\\
            \end{array} \right. 
        \end{equation*}
\end{proof}

Following \cite[Section 5]{Che22}, we use the rest of this section to show that, in the time interval $[\zeta_K,T_K]$, the genealogy does not change in the limit $K \to \infty$.
Since the necessary changes are more subtle, we will provide the complete proofs.
To make the computations more clear, we consider the case $\theta^+=\theta^-=0$, but the proofs follow through also in the general case.

\begin{lemma}\label{lem:Lemma5.1}
Let $T_K$ be a stopping time w.r.t.~the natural filtration of $\mathcal{N}_K$ and write $\tau_K := \inf\{t: N_K(t) \geq K\}$ for the time that the logistic branching reaches its carrying capacity for the first time.
Then, for $\widehat{T}_K := T_K\wedge \tau_K$, there exists a constant $c > 0$ independent of $K$ s.t.~for all $s \geq 0$,
\[
\mathbb{E}\left[ \Big( F_K^u(\widehat{T}_K) - F_K^u(s\wedge \widehat{T}_K)\Big)^2\right] \leq c\int_s^{+\infty} \mathbb{E}\left[\dfrac{\1_{\{N_K(t) > 0, t < \widehat{T}_K\}}}{N_K(t)}\right]\dx t.
\]
\end{lemma}
\begin{proof}
The function $G^u:\mathfrak{T}\times\mathfrak{T} \rightarrow [0,1]$ given by $G^u(\emptyset) := 0$ and
\[
G^u(\mathcal{N}) := \dfrac{\abs{\{v\in \mathcal{N}\;:\; u\prec v\}} }{\abs{\mathcal{N}}}\qquad \text{ for }\mathcal{N}\neq \emptyset
\]
is bounded and satisfies $F_K^u(t) = G^u(\mathcal{N}_K(t))$. Writing $\mathcal{L}_K$ for the infinitesimal generator of $\mathcal{N}_K$, this yields

\begin{equation*}
    \begin{split}
        &\E\left[\left(F^u_K(\widehat{T}_K)-F^u_K(s\wedge \widehat{T}_K)\right)^2\right]\\&= \int_s^{+\infty} \left(\E\left[\1_{\{t<\widehat{T}_K\}} \mathcal{L}_K(G^u)^2(\mathcal{N}_K(t)) \right]-2\E\left[\1_{\{t<\widehat{T}_K\}} G^u(\mathcal{N}_K(s))\mathcal{L}_K G^u(\mathcal{N}_K(t)) \right]\right) \dx t.
    \end{split}
\end{equation*}
Since $0\leq G^u(\mathcal{N}_K(s))\leq 1$, it is sufficient to show that both $\abs{\mathcal{L}_KG^u(\mathcal{N})}$ and $\abs{\mathcal{L}_K(G^u)^2(\mathcal{N})}$ can be bounded by a constant times $\1_{\{\abs{\mathcal{N}}\neq 0\}}/\abs{\mathcal{N}}$. We use the following notation: we write $\bar{u}$ for the parent of $u$, $n:=\abs{\mathcal{N}}$, $n^\pm:=\abs{\mathcal{N}^\pm}$, $g := G^u(\mathcal{N})$ and
\begin{equation*}
    g^+:=\frac{\abs{\{v\in\mathcal{N}^+\,:\, u\prec v\}}}{\abs{\mathcal{N}^+}} \quad \quad g^-:=\frac{\abs{\{v\in\mathcal{N}^-\,:\, u\prec v\}}}{\abs{\mathcal{N}^-}}.
\end{equation*}
As before, if $n=0$, $n^+=0$ or $n^-=0$, we set  $g=0$, $g^+=0$ or $g^-=0$, respectively.  
The generator $\mathcal{L}_K$ applied to $G^u$ is given by
\begin{equation*}
\begin{split}
\mathcal{L}_K G^u(\cdot)(\mathcal{N}) &= \frac{1}{n + 1}\sum_{i \geq 1}\left(\mu_K^+(i)\1_{\{\bar{u}\in \mathcal{N}^+\}}+\mu_K^-(i)\1_{\{\bar{u}\in \mathcal{N}^-\}}\right)\\
&\;\quad+ \sum_{i \geq 1}\mu_K^+(i) n^+ g^+\left(\frac{ng+i-1}{n+i-1}-g\right) + \sum_{i \geq 1}\mu_K^+(i) n^+ (1-g^+)\left(\frac{ng}{n+i-1}-g\right)\\
&\;\quad+\sum_{i \geq 1}\mu_K^-(i) n^- g^-\left(\frac{ng+i-1}{n+i-1}-g\right) + \sum_{i \geq 1}\mu_K^-(i) n^- (1-g^-)\left(\frac{ng}{n+i-1}-g\right)\\
&\;\quad + \left(\mu^+_K(0) + n\frac{\mathfrak{m}}{K}\right)\left(n^+g^+\left(\frac{ng-1}{n-1}-g\right) + n^+(1-g^+)\left(\frac{ng}{n-1}-g\right)\right)\1_{\{n > 1\}}\\
&\;\quad + \left(\mu^-_K(0) + n\frac{\mathfrak{m}}{K}\right)\left(n^-g^-\left(\frac{ng-1}{n-1}-g\right)n^-(1-g^-)\left(\frac{ng}{n-1}-g\right)\right)\1_{\{n > 1\}}\\
&\;\quad+ \left(\left(\mu^+_K(0) + n\frac{\mathfrak{m}}{K}\right)n^+g^+\left( 0-g\right) + \left(\mu^-_K(0) + n\frac{\mathfrak{m}}{K}\right)n^-g^-\left(0-g\right)\right)\1_{\{n = 1\}}.
\end{split}
\end{equation*}
For $n=0$, the above expression vanishes; for $n\geq 1$ it can be rewritten as

\begin{equation*}
\begin{split}
\mathcal{L}_K G^u(\cdot)(\mathcal{N}) = &\frac{1}{n + 1}\sum_{i \geq 1}\left(\mu_K^+(i)\1_{\{\bar{u}\in \mathcal{N}^+\}}+\mu_K^-(i)\1_{\{\bar{u}\in \mathcal{N}^-\}}\right)\\
&+ \sum_{i \geq 1}(i-1)\mu_K^+(i) \frac{n}{n+i-1} \left( \frac{n^+ g^+}{n}(1-g) - \frac{n^+ (1-g^+)}{n}g\right)\\
&+ \sum_{i \geq 1}(i-1)\mu_K^-(i) \frac{n}{n+i-1} \left( \frac{n^- g^-}{n}(1-g) - \frac{n^- (1-g^-)}{n}g\right)\\
& + \left(\mu^+_K(0) + n\frac{\mathfrak{m}}{K}\right)\frac{n}{n-1}\left( \frac{n^+g^+}{n}(1-g) - \frac{n^+(1-g^+)}{n}g\right)\1_{\{n > 1\}}\\
& + \left(\mu^-_K(0) + n\frac{\mathfrak{m}}{K}\right)\frac{n}{n-1}\left( \frac{n^-g^-}{n}(1-g) - \frac{n^-(1-g^-)}{n}g\right)\1_{\{n > 1\}}\\
&+ \left(\left(\mu^+_K(0) + n\frac{\mathfrak{m}}{K}\right)n^+ g^+\left( 0-g\right) + \left(\mu^-_K(0) + n\frac{\mathfrak{m}}{K}\right)n^-g^-\left(0-g\right)\right)\1_{\{n  = 1\}}.
\end{split}
\end{equation*}
 For $n=1$, the expression is uniformly bounded in absolute value.
In particular, $$\mathcal{L}_KG^u(\mathcal{N})\mathds{1}_{n = 1} \leq c \frac{1}{n}\,\1_{\{n = 1\}},$$ for some constant $c>0$. 
For $n>1$, using the fact that $n^-g^-=ng-n^+g^+$, the expression simplifies to

\begin{equation*}
\begin{split}
  \mathcal{L}_K G^u(\cdot)(\mathcal{N}) =& \frac{1}{n+1}
  \sum_{i \geq 1}\left(\mu_K^+(i)\1_{\{\bar{u}\in \mathcal{N}^+\}}+\mu_K^-(i)\1_{\{\bar{u}\in \mathcal{N}^-\}}\right)\\
  &+C_g\left(\sum_{i \geq 0}(i-1)(\mu^+_K(i)-\mu^-_K(i))\frac{n}{n+i-1}\right),
\end{split}
\end{equation*}
where 
\begin{equation*}
    C_g:=\left(\frac{n^+g^+}{n}(1-g)-\frac{n^+(1-g^+)}{n}g\right)\in [-1,1].
\end{equation*}
Using the uniform convergence assumption \eqref{eq:convergence-measure}, the first sum is bounded by a constant times $1/n$. For the second sum, we note that 
\begin{equation*}
    \frac{n}{n+i-1}=1-\frac{i-1}{n+i-1},
\end{equation*}
so that
\begin{equation}\label{eq:generator-bound}
    \begin{split}
        &\left|\sum_{i \geq 0}(i-1)(\mu^+_K(i)-\mu^-_K(i))\frac{n}{n+i-1}\right|\\
        &\leq \left\vert \sum_{i \geq 0}(i-1)(\mu^+_K(i)-\mu^-_K(i))\right\vert +\sum_{i \geq 0}\frac{(i-1)^2}{n+i-1}\left(\mu_K^+(i)+\mu_K^-(i)\right)\\
        &\leq \frac{\abs{s^+-s^-}}{K}+\text{o}\left(\frac{1}{K}\right)+\frac{2(\mathfrak{v}^+ +\mathfrak{v}^- +\text{o}(1))}{n},
    \end{split}
\end{equation}
where we used that 
\begin{equation*}
    \frac{1}{n+i-1}\leq \frac{2}{n} \quad \text{for} \quad n\geq 2.
\end{equation*}

Since we are interested in the process up to the stopping time $\widehat{T}_K$, we have $n \leq K$. Hence the previous expression can be bounded by a constant times $1/n$.

Similarly, we have that, for $n\geq 2$,
\begin{equation}\label{eq:generator-second-moment}
    \begin{split}
        \mathcal{L}_K(G^u)^2(\mathcal{N})=&\left(\frac{1}{n+1}\right)^2\sum_{i \geq 1}\left(\mu_K^+(i)\1_{\{\bar{u}\in \mathcal{N}^+\}}+\mu_K^-(i)\1_{\{\bar{u}\in \mathcal{N}^-\}}\right)\\
        &+ g\,C_g\sum_{i \geq 0}(i-1)(\mu^+_K(i)-\mu^-_K(i)) \frac{n}{n+i-1}\left(2-\frac{i-1}{n+i-1}\right)\\
        &+n(1-g)\frac{n^+g^+}{n}\sum_{i \geq 0}\left(\frac{i-1}{n+i-1}\right)^2 (\mu^+_K(i)+\mu^-_K(i))\\
        &+\frac{n}{n-1}\frac{\mathfrak{m}}{K}g(1-g).
    \end{split}
\end{equation}
The first term can be bounded by a constant times $1/n^2\leq 1/n$. For the second term, we note that
\begin{equation*}
    \begin{split}
        &\left|\sum_{i \geq 0}(i-1)(\mu^+_K(i)-\mu^-_K(i)) \frac{n}{n+i-1}\left(2-\frac{i-1}{n+i-1}\right)\right|\\
        &\leq 2 \left|\sum_{i \geq 0}(i-1)(\mu^+_K(i)-\mu^-_K(i)) \frac{n}{n+i-1} \right|\\
        &\quad + \sum_{i\geq 0}(i-1)^2(\mu^+_K(i)+\mu^-_K(i))\frac{n}{(n+i-1)^2},
    \end{split}
\end{equation*}
which can be bounded in the same way as \eqref{eq:generator-bound}. The second to last term of \eqref{eq:generator-second-moment} can be bounded by a constant times $1/(n+i-1)\leq 2/n$ and for the last term we recall that we work on $\{n\leq K\}$.

For $n\leq 1$, one can get the correct bound as before.

 \end{proof}
 Recall that for $\beta\in (0,1)$
\[
T_K^\beta = \inf\{t\geq 0\;:\; N_K(t) \geq K - K^\beta\} \leq \tau_K,
\]
so that $\widehat{T}_K^\beta = T_K^\beta$ in the above notation.
\begin{lemma}\label{lem:Lemma 5.2}
    There exists a constant $b = b(\beta) > 0$ independent of $K$ such that
    \[
    \mathbb{E}\left[\dfrac{\mathds{1}_{N_K(t) > 0, t < {T}_K^\beta}}{N_K(t)}\right] \leq bt^{-\frac 1{1-\beta}},
    \]
    for all $t > 0$.
\end{lemma}
\begin{proof}
First, define
\[
    \bar\mu_K^{n^+, n^-}(i) := \begin{cases}
        n^+\mu_K^+(i) + n^-\mu_K^-(i) & \text{ if } i\geq 1,\\
        n^+\mu_K^+(0) + n^-\mu_K^-(0) + n^2\dfrac{\mathfrak{m}}{K}& \text{ if } i= 0,
    \end{cases}
\]
so that the overall growth rate of a population $(n^+, n^-)$ is given by
\begin{align*}
    \sum_{i\geq 0} (i-1)\bar\mu_K^{n^+, n^-}(i) &= n\mathfrak{m} + \dfrac{n^+\mathfrak{s}^+}{K} + \dfrac{n^-\mathfrak{s}^-}{K} + n\,\text{o}\left(\dfrac{1}{K}\right) - n^2\dfrac{\mathfrak{m}}{K}\\
    &= n\left( \mathfrak{m}\left(1 - \frac{n}{K}\right) + \mathcal{O}\left(\frac{1}{K}\right)\right).
\end{align*}
    Set $M_K := \{1, 2, \dots, K-K^\beta\}$.
    From \eqref{eq:marta_superlinear_population} and the following discussion, we have
    \[
\mathfrak{m}\left(1 - \dfrac{n}{K}\right) \geq \dfrac{c}{n^{1-\beta}},
    \]
    for all $n\in M_K$, for some constant $c$ independent of $n$ and $K$.
    To include the $\mathcal{O}$-term, it suffices to note that $n\in M_K$ implies $n\leq K$, so that $\frac{1}{K}$ is of lower order.
    
    Since $I:\mathcal{N}\mapsto \frac{\mathds{1}_{\vert \mathcal{N}\vert \in M_K}}{\vert \mathcal{N}\vert}$ is bounded, we may conclude that
    \[
        \partial_t\mathbb{E}\left[\dfrac{\mathds{1}_{N_K(t) > 0, t < {T}_K^\beta}}{N_K(t)}\right]= \partial_t\mathbb{E}\left[ I(\mathcal{N}_K(t\wedge T_K^\beta)) \right]
        = \mathbb{E}\left[ \mathds{1}_{t<T_K^\beta}\mathcal{L}_K I(\mathcal{N}_K(t))\right].
    \]
    Now, for $n = \vert\mathcal{N}\vert \in M_K \setminus \{1\}$, we have
    \begin{align*}
        \mathcal{L}_K I(\mathcal{N})=& \sum_{i\geq 0} \bar\mu_K^{n^+, n^-}(i)\left(\dfrac{\mathds{1}_{n+i-1\in M_K}}{n + i - 1} - \dfrac{1}{n}\right) \\
        =& -\dfrac{1}{n}\sum_{i\geq 0} \bar\mu_K^{n^+, n^-}(i)\left(\dfrac{(i-1)\mathds{1}_{n+i-1\in M_K}}{n + i - 1} + \mathds{1}_{n + i -1 \not\in M_K}\right)\\
        =& -\dfrac{1}{n}\sum_{i\geq 0} \bar\mu_K^{n^+,n^-}(i)\left( \dfrac{i-1}{n+i-1} + \dfrac{n}{n+i-1}\mathds{1}_{n+i-1\not\in M_K}\right)\\
        \leq& -\dfrac{1}{n}\sum_{i\geq 0} \bar\mu_K^{n^+, n^-}(i)\dfrac{i-1}{n+i-1},
    \end{align*}
    where we used the fact that $1=\mathds{1}_{n+i-1\in M_K}+\mathds{1}_{n+i-1\notin M_K}$
    Next, we note that, for $n\in M_K \setminus \{1\}$,
    
    \[
\begin{split}
    \dfrac{1}{n}\sum_{i\geq 0} \bar\mu_K^{n^+, n^-}\!\!(i) \left\vert\dfrac{i-1}{n+i-1} - \dfrac{i-1}{n}\right\vert &\leq \dfrac{1}{n}\sum_{i\geq 0} \bar\mu_K^{n^+, n^-}\!\!(i) \dfrac{(i-1)^2}{n(n+i-1)}\\
    &\leq  \dfrac{1}{n^3}\sum_{i\geq 0} (i-1)^2\bar\mu_K^{n^+, n^-}\!\!(i)\\
    &\leq  \dfrac{1}{n^2}\sum_{i\geq 0} (i-1)^2 \big(\mu_K^+(i) + \mu_K^-(i)\big) + \dfrac{\mathfrak{m}}{nK},
\end{split}
    \]
    which is uniformly bounded by a constant times $n^{-2}$.
    This allows us to write
    \[
\begin{split}
    &\sum_{i\geq 0} \bar\mu_K^{n^+, n^-}\!\!(i)\dfrac{i-1}{n(n+i-1)}\\
    \geq& \left(\dfrac{1}{n^2}\sum_{i\geq 0} (i-1)\bar\mu_K^{n^+, n^-}\!\!(i) - dn^{-2}\right) \\
    \geq & \left(cn^{\beta - 2} -  dn^{-2}\right),
\end{split}
    \]
    for some constants $c,d > 0$ independent of $K$.
    As before, since the quantity under consideration is positive and the $n^{-2}$-correction is of lower order, there exists some constant $\gamma > 0$ such that
    \[
\sum_{i\geq 0} \bar\mu_K^{n^+,n^-}(i)\dfrac{i-1}{n(n+i-1)}\mathds{1}_{n\in M_K} \geq \gamma n^{\beta-2}.
    \]
    We finally notice that the same inequality holds for $n=1$.
    Plugging this back into the generator expression, we obtain
    \[
    \begin{split}
\partial_t\mathbb{E}\left[ \dfrac{\mathds{1}_{N_K(t) > 0, t < T_K^\beta}}{N_K(t)}\right] &\leq - \gamma\mathbb{E}\left[\dfrac{\mathds{1}_{N_K(t) > 0, t < T_K^\beta}}{N_K(t)^{2-\beta}}\right] \\
&\leq -\gamma\mathbb{E}\left[\dfrac{\mathds{1}_{N_K(t) > 0, t < T_K^\beta}}{N_K(t)}\right]^{2-\beta},
\end{split}
    \]
    so that
    \[
\mathbb{E}\left[ \dfrac{\mathds{1}_{N_K(t) > 0, t < T_K^\beta}}{N_K(t)}\right] \leq \left( \mathbb{E}\left[\dfrac{N_K(0)\in M_K}{N_K(0)}\right]^{-(1-\beta)} + (1-\beta)\gamma t\right)^{-\frac{1}{1-\beta}}.
    \]
\end{proof}

We are now ready to prove the main result of the Section.

\begin{proof}[Proof of \Cref{thm:application_of_comparison}]
The previous two results yield
\[
\begin{split}
    \mathbb{E}\left[ \left(F_K^g(T_K^\beta) - F_K^g(\zeta_K\wedge T_K^\beta)\right)^2\right] &\leq c\int_{\zeta_K}^{+\infty} \mathbb{E}\left[\dfrac{\mathds{1}_{N_K(t) > 0, t < T_K^\beta}}{N_K(t)}\right]\dx t\\
    &\leq c\int_{\zeta_K}^{+\infty} bt^{-\frac{1}{1-\beta}}\dx t,
\end{split}
\]
which goes to zero as $\zeta_K \to +\infty$ with $K\to +\infty$.
The result then follows from \Cref{cor:cheeky_zeta_K_min_T_K}.
\end{proof}

\section{Interlude: the Katzenberger method}\label{sec:Katzenberger-method}

To study the frequency process at carrying capacity and prove \Cref{theo:convergence_at_carrying_capacity}, we will make heavy use of the little known results from \cite{Kat91} on solutions to SDEs forced onto a manifold by a large drift.
For the convenience of the reader, in Section \ref{sec:Katzenberger}, we give an overview of the main result from \cite{Kat91}.
In Section \ref{sec:Kat-Markov}, we deduce a simpler result in the context of Markov processes which we then use in the proof of \Cref{theo:convergence_at_carrying_capacity}.
For more details on the motivation for the Katzenberger method, see \cite[Appendix B]{Kern25}.

Please note that the notation in this section follows \cite{Kat91} and does not reflect the exact notation of the remainder of this paper.
In particular, we will use the following.
We write $V\subset\!\subset U$ whenever $V$ is a \emph{compact} subset of $V$, and $\mathring{V}$ for the interior of $V$.
For a differentiable $F:U\rightarrow \mathbb R^d$, we write $\partial F(x)\in \mathbb R^{d\times d}$ for the Jacobian of $F$ at $x$ as well as $\partial_{ij} F$, $1\leq i,j\leq d$ for the second partial derivatives of $F$ whenever they exist.
We write $\mathbb{D}([0,+\infty); \mathbb{R}^d)$ for the space of càdlàg functions $[0,+\infty)\rightarrow\mathbb{R}^d$, endowed with the usual Skorokhod topology.
If $X$ is a càdlàg process, $X(t-)$ denotes the left limit of $X$ in $t$ and $\Delta X(t) := X(t) - X(t-)$ denotes the jump of $X$ at $t$.
If $X$ and $Y$ are semimartingales, we write $[X,Y]$ (and $\langle X, Y\rangle$) for the (predictable) quadratic covariation.
Finally, for any sequence of random variables, we write $X_n \Rightarrow X$ to mean that $(X_n)_{n\in \N}$ converges in distribution to $X$, as $n\to \infty$.

\subsection{SDEs forced onto a manifold}\label{sec:Katzenberger}

Let $U\subseteq \mathbb R^d$ be an open set, $F:U\rightarrow\mathbb R^d$ a continuous vector field, and define
\[
\Gamma := F^{-1}(0) = \{ x\in U\;:\; F(x) = 0\}.
\]
For every $n\in\mathbb N$, let $(\Omega^n, \mathcal{F}^n, (\mathcal{F}_t^n)_{t\geq 0}, \mathbb P^n)$ be a filtered probability space, let $Z_n$ be an $\mathbb R^m$-valued càdlàg $(\mathcal{F}_t^n)_{t\geq 0}$-semimartingale with $Z_n(0) = 0$, and let $A_n$ be a real-valued càdlàg $(\mathcal{F}_t^n)_{t\geq 0}$-adapted non-decreasing process with $A_n(0) = 0$.
Let $\sigma:U\rightarrow \mathbb R^{d\times m}$ be continuous, and $\sigma_n: U\rightarrow \mathbb R^{d\times m}$ be continuous and such that $\sigma_n\to \sigma$ uniformly on compact subsets of $U$.
Suppose that $X_n$ is an $\mathbb R^d$-valued càdlàg $(\mathcal{F}_t^n)_{t\geq 0}$-semimartingale that satisfies
\begin{equation}\label{eq:X_n-Katz}
X_n(t) = X_n(0) + \int_0^t \sigma_n\big(X_n(s-)\big)\dx Z_n(s) + \int_0^t F\big(X_n(s-)\big)\dx A_n(s)
\end{equation}
for all $t\leq \tau_n(V)$ and all compacts $V\subset\!\subset U$, where
\[
\tau_n(V) := \inf\{t\geq 0\;:\; X_n(t-)\not\in \mathring{V}\text{ or }X_n(t)\not\in\mathring{V}\}.
\]
We additionally assume the following asymptotic continuity
\begin{equation}\label{eq:katzenberger_asymptotic_cont}
\sup_{0\leq t\leq T\wedge\tau_n(V)} \Delta A_n(t) + \vert \Delta Z_n(t)\vert \Rightarrow 0,
\end{equation}
with $n\to\infty$ for every $T > 0$ and every compact $V\subset\!\subset U$.

We write
\[
D(0) := \{ z\in \mathbb C\;:\; \Re(z) < 0\}
\]
for the left half of the complex plane.
Furthermore, define the flow
\[
\psi(x,t) = x + \int_0^t F\big(\psi(x,s)\big)\dx s
\]
and let
\[
U_\Gamma := \left\{ x\in U\;:\; \lim_{t\to+\infty} \psi(x,t)\text{ exists and is in }\Gamma\right\},
\]
denote the domain of attraction of $\Gamma$ under $\psi$.
For $x\in U_\Gamma$, we set 
\[
\Phi(x) := \lim_{t\to+\infty} \psi(x,t).
\]

\begin{assumption}[Regularity of the deterministic system]\label{ass:katz_regularity_deter}
Assume that $F$ is twice continuously differentiable and $\Gamma$ is a $\mathcal{C}^2$-manifold of $U$.
If $p$ is the dimension of the manifold $\Gamma$, then, for every $y\in \Gamma$, the matrix $\partial F(y)$ has $d-p$ eigenvalues in $D(0)$.
Finally, assume that $\Phi$ is twice continuously differentiable.
\end{assumption}

Note that the final assumption on $\Phi$ is implied by the other assumption together with the additional assumption that the second derivative of $F$ is Lipschitz continuous, see \cite[Corollary 3.6]{Kat91}.

\begin{assumption}[Infinite drift]\label{ass:katz_infinite drift}
For all $\epsilon,T > 0$ and $V\subset\!\subset U$, we have
\[
\inf_{0\leq t\leq T\wedge\tau_n(V)-\epsilon} \Big( A_n(t+\epsilon) - A_n(t)\Big) \underset{n\to \infty
}{\Longrightarrow} +\infty,
\]
where the infimum over the empty set is defined to be $+\infty$.
\end{assumption}

For $\delta\in (0,+\infty]$, define
\[
h_\delta(r) := \begin{cases}
    0 & \text{ if } r\leq \delta,\\
    1 - \delta/r & \text{ if } r\geq \delta,
\end{cases}
\]
as well as the map $J_\delta : \mathbb D([0,+\infty); \mathbb R^m) \rightarrow \mathbb D([0,+\infty);\mathbb R^m)$ through
\[
J_\delta(g)(t) := \sum_{0\leq s\leq t} h_\delta\big(\vert \Delta g(s)\vert\big) \Delta g(s).
\]
The function $g\mapsto g - J_\delta(g)$ simply truncates the jumps of $g$ at $\delta$. Here and in the following, $\vert x\vert$ denotes the standard euclidean norm.

\begin{assumption}[Regularity of the driving noise]\label{ass:katz_regularity_noise}
    For any $V\subset\!\subset U$, set $Y_n := Z_n^{\tau_n(V)}$. 
    Assume that there exists some $\delta\in(0,+\infty]$ and a sequence of stopping times $(\widehat{\tau}_n^k)_{k\in\mathbb N}$ for every $n\in\mathbb N$ as well as a decomposition of $Y_n - J_\delta(Y_n)$ into a local martingale $M_n$ plus a process $F_n$ of finite variation such that $\mathbb P^n(\widehat{\tau}_n^k \leq k)\leq 1/k$,
    \[
        \Big([M_n]\left(t\wedge\widehat{\tau}_n^k\right) + T_{t\wedge\widehat{\tau}_n^k}(F_n)\Big)_{n\in\mathbb N}
    \]
    is uniformly integrable for every $t\geq 0$ and $k\in\mathbb N$, and 
    \[
    \lim_{\gamma\downarrow 0}\limsup_{n\to\infty} \mathbb P^n\left( \sup_{0\leq t\leq T} \Big(T_{t+\gamma}(F_n) - T_t(F_n)\Big)> \epsilon\right) = 0
    \]
    for all $\epsilon,T > 0$.
    Here, $T_t(F_n)$ denotes the total variation of $F_n$ over the time interval $[0,t]$.
\end{assumption}

Whereas Assumptions \ref{ass:katz_regularity_deter} and \ref{ass:katz_infinite drift} can be linked directly to the behaviour of the deterministic system, \Cref{ass:katz_regularity_noise} is used to ensure that the stochastic integrals are well-behaved.
Indeed, if we take $Y_n$ to satisfy the above (cf.~\cite[Condition 4.2]{Kat91}) and also assume it to be relatively compact, then $\left(Y_n, \int H_n\dx Y_n\right)_{n\in\mathbb N}$ is relatively compact for a wide range of adapted processes $H_n$, see \cite[Proposition 4.4]{Kat91}.
See also \cite{kurtz1991protter} for more criteria on weak convergence of stochastic integrals.

\begin{theorem_internal}[{\cite[Theorem 6.2]{Kat91}}]\label{theo:katz_convergence}
Let $(X_n)_{n \in \N}$ be a sequence of processes of the form \eqref{eq:X_n-Katz} satisfying the asymptotic continuity \eqref{eq:katzenberger_asymptotic_cont} as well as $X_n(0) \Rightarrow X_n(0)\in\Gamma$.
Under Assumptions 1--3, the following holds true for every $V\subset\!\subset U$:
\begin{enumerate}
    \item The sequence $\left(X_n^{\tau_n(V)}, Z_n^{\tau_n(V)}, \tau_n(V)\right)_{n\in\mathbb N}$ is relatively compact in the natural product space ${\mathbb D([0,+\infty); \mathbb R^d)\times \mathbb{D}([0,+\infty);\mathbb R^m)\times [0,+\infty]}$.
    \item For any limit point $(X, Z,\tau)$ of the above sequence, we have that 
    \begin{enumerate}
        \item $(X,Z)$ is a continuous semimartingale,
        \item $X(t) \in\Gamma$ for every $t\geq 0$ a.s.,
        \item $\tau \geq \inf\{t \geq 0\;:\; X(t) \not\in \mathring{V}\}$ {a.s.},
        \item one has
        \[
        \begin{split}
            X(t) &= X(0) + \int_0^{t\wedge\tau} \partial\Phi \cdot \sigma\big( X(s)\big)\dx Z(s) + \dfrac{1}{2}\sum_{i,j=1}^d\sum_{k,\ell=1}^m \int_0^{t\wedge\tau} \partial_{ij}\Phi \sigma^{ik}\sigma^{j\ell} \big(X(s)\big) \dx[Z^k, Z^\ell](s).
        \end{split}
        \]
    \end{enumerate}
\end{enumerate}
\end{theorem_internal}

For an extension of \Cref{theo:katz_convergence} to the case where the initial condition converges to a point that lies in the domain of attraction of $\Gamma$, see \cite[Theorem 6.3]{Kat91}.

\subsection{The Katzenberger method}\label{sec:Kat-Markov}

We now derive a simpler version of \Cref{theo:katz_convergence} in the context of Markov processes forced onto a manifold.
Since, for the proof of \Cref{theo:convergence_at_carrying_capacity}, we are interested in Markov processes that behave like diffusions on manifolds, the main result of this section should be viewed as a generalisation of the results in \cite{funaki1993drift} which is also a special case of \cite{Kat91}.

We use the notation from the previous section. In the following, suppose that, for all $n\in \N$, $X_n$ is a Markov process stopped upon leaving $U$ with generator
\begin{equation}\label{eq:katzenberger_approximate_generator_decomposition}
\mathcal{C}_n f(x) = \Big(B_n f(x) + \kappa_n F\cdot \nabla f(x)\Big)\mathds{1}_{x\in U},
\end{equation}
where $\kappa_n \to +\infty$ and the $B_n$ are operators with common domain $\mathcal{D}$ containing the projections 
\[
I_i:  U\longrightarrow \mathbb R,\quad x\mapsto x_i,
\]
for all $1\leq i\leq d$ as well as products of the form $I_iI_j$, $1\leq i,j\leq d$.
Finally, assume that $B_n f \to Bf$ uniformly on compacts $V\subset\!\subset U$ for all $f$ of the above form, where
\[
Bf = b\cdot \nabla f + \dfrac{1}{2}\sum_{i,j=1}^d c_{ij}\partial_{ij}f,
\]
is the generator of a diffusion with $b:U\to \mathbb R^d$ and $c:U\to \mathbb R^{d\times d}$ measurable and locally bounded.
In this context, it is natural to assume that
\begin{equation}\label{eq:katz_asymp_cont_markov}
\sup_{0\leq t\leq T\wedge\tau_n(V)} \vert \Delta X_n(t)\vert \underset{n\to\infty}{\Longrightarrow} 0,
\end{equation}
for all $T\geq 0$ and $V\subset\!\subset U$.

From $I_i\in \mathcal{D}$ for $1\leq i\leq d$, we conclude that $X_n$ is a semimartingale.
Setting $A_n(t) := \kappa_n t$ and $Z_n(t) := X_n(t) - X_n(0) - \int_0^t F(X_n(s))\dx A_n(s)$, we obtain the representation
\[
X_n(t) = X_n(0) + Z_n(t) + \int_0^t F\big(X_n(s)\big)\dx A_n(s).
\]
We assume \Cref{ass:katz_regularity_deter} to hold. \Cref{ass:katz_infinite drift} is verified by definition of $A_n$.
Furthermore, the asymptotic continuity is clear for $A_n$ and follows for $Z_n$ from \eqref{eq:katz_asymp_cont_markov}.
We therefore concentrate on \Cref{ass:katz_regularity_noise}.
Write $I:x\mapsto x$ for the identity on $\mathbb R^d$.
By Dynkin's formula for $X_n$, we may decompose $Z_n$ into the martingale
\[
M_n(t) = Z_n(t) - \int_0^t B_n I\big(X_n(s)\big)\dx s,
\]
plus the process of finite variation
\[
F_n(t) = \int_0^t B_n I\big(X_n(s)\big)\dx s.
\]
Note that the corresponding total variation process is given by
\[
T_t(F_n) = \int_0^t \Vert B_n I\big(X_n(s)\big)\Vert_1\dx s.
\]
In particular, since
\[
B_n I \underset{n\to\infty}{\longrightarrow} BI = b,
\]
uniformly on compacts, and $b$ is bounded on compacts, we may conclude that $t\mapsto T_{t\wedge\tau_n(V)}(F_n)$ is asymptotically continuous and that $\left(T_{t\wedge\tau_n(V)}(F_n)\right)_{n\in\N}$ is uniformly integrable for every $t\geq 0$.
In particular, $Z_n(t\wedge\tau_n(V))$ satisfies \Cref{ass:katz_regularity_noise} as soon as
\[
[M_n](t\wedge\tau_n(V)) = [Z_n](t\wedge\tau_n(V)) = [X_n](t\wedge\tau_n(V))
\]
is uniformly integrable in $n$.
We will assume this to hold for every $V\subset\!\subset U$ for the moment.

Then, \Cref{theo:katz_convergence} yields for every $V\subset\!\subset U$ that the triple $\left(X_n^{\tau_n(V)}, Z_n^{\tau_n(V)}, \tau_n(V)\right)$ is relatively compact and that any limit point $(X,Z,\tau)$ satisfies the equation
\[
\begin{split}
X(t) &= X(0) + \int_0^{t\wedge\tau} \partial\Phi\big(X(s)\big)\dx Z(s) \\
&\hspace{1.25cm} + \dfrac{1}{2}\sum_{i,j=1}^d \int_0^{t\wedge\tau} \partial_{ij}\Phi\big(X(s)\big)\dx [Z^i, Z^j](s),
\end{split}
\]
where $\tau \geq \tau(V) := \inf\{t\geq 0\;:\; X(t)\not\in \mathring{V}\}$.
To identify the limit, we may note that, by the Continuous Mapping Theorem, $M^{\tau_n(V)}_n$ converges jointly with the rest to some process $M$.
Following \cite[Theorem 5.1]{yurachkivsky2011martingale}, we will assume that
\[
\mathbb E_n\left[ \sup_{0\leq t\leq T\wedge\tau_n(V)} \vert \Delta X_n(t)\vert^2\right] \underset{n\to\infty}{\longrightarrow} 0,
\]
to conclude that $M$ is a continuous local martingale and $\langle M^i_n, M_n^j\rangle \Rightarrow \langle M^i, M^j\rangle$ jointly for all indices and jointly with $(X,Z)$, so that
\[
Z(t) = M(t) + \int_0^t b\big(X(s)\big)\dx s,
\]

and
\[
\langle M^i, M^j\rangle(t) = \int_0^t c_{ij}\big(X(s)\big)\dx s,
\]
with $[Z^i, Z^j] = [M^i, M^j] = \langle M^i, M^j\rangle$ for all $1\leq i,j\leq d$.

In particular, 
\[
\begin{split}
    X(t) &= X(0) + \int_0^{t\wedge\tau} \partial\Phi\big(X(s)\big)\dx M(s)\\
    &\quad + \int_0^{t\wedge\tau} \!\!\left(\!\partial\Phi\big(X(s)\big)b\big(X(s)\big) + \dfrac{1}{2}\!\sum_{i,j=1}^d \partial_{ij}\Phi \big(X(s)\big)c_{ij}\big(X(s)\big)\!\right)\!\!\!\dx s.
\end{split}
\]

Finally, Itô's formula yields that $(X,\tau(V))$ is a solution to the stopped martingale problem for
\begin{equation}\label{eq:katzenberger_limiting_generator}
\begin{split}
    \mathcal{C} f(x) &= \left( \partial\Phi(x)b(x) + \dfrac{1}{2}\sum_{i,j=1}^d \partial_{ij}\Phi(x)c_{ij}(x)\right)^T\cdot\nabla f(x) \\
    &\qquad + \dfrac{1}{2}\sum_{i,j,k,\ell=1}^d \partial_k\Phi_i(x)\partial_\ell\Phi_j c_{k\ell}(x)\partial_{ij}f(x).
\end{split}
\end{equation}

\begin{theorem_internal}[Katzenberger method]\label{theo:katzenberger_method}
Let \Cref{ass:katz_regularity_deter} hold.
    Consider an operator $\mathcal{C}_n$ of the form \eqref{eq:katzenberger_approximate_generator_decomposition} and assume that $X_n$ is a solution to the martingale problem for $\mathcal{C}_n$ satisfying
    \[
    \mathbb E_n\left[ \sup_{0\leq t\leq T\wedge\tau_n(V)} \vert \Delta X_n(t)\vert^2\right] \underset{n\to\infty}{\longrightarrow} 0,
    \]
    for all $T\geq 0$ and $V\subset\!\subset U$. Assume additionally that $[X_n](t\wedge\tau_n(V))$ is uniformly integrable for all $V\subset\!\subset U$. 
    Then, the sequence $(X_n)_{n\in\N}$ is relatively compact.
    If $X$ is a limit point and $\tau(V) := \inf\{t\geq 0\;:\; X(t)\not\in \mathring{V}\}$, then $(X,\tau(V))$ is a solution to the stopped martingale problem for $\mathcal{C}$ in \eqref{eq:katzenberger_limiting_generator}.
\end{theorem_internal}

\section{The frequency process at carrying capacity}\label{sec:limit-carrying-capacity}

In this section, we prove the forward-in-time convergence of the frequency process (\Cref{theo:convergence_at_carrying_capacity}) using the Katzenberger method developed in the previous section.
We will assume throughout this section that there exists $\eta > 0$ such that 
\begin{equation}\label{eq:eta-assumption}
  \sum_{i\geq 0} i^{2+\eta} \mu_K^\pm(i) \text{  is uniformly bounded in $K$.}
\end{equation}

First, recall that
\[
X_K^\pm(t) := \dfrac{N_K^\pm(Kt)}{K}
\]
defines a Markov process $X_K = (X_K^+, X_K^-)$ with generator 
\begin{equation}\label{eq:marta_generator_of_X_K}
    \begin{split}
    &A_Kf(x)=\sum_{i\geq 0} K^2x^+\mu_K^+(i)\left(f\left(x+\frac{i-1}{K}\delta_+\right)-f(x)\right)\\&+\sum_{i\geq 0} K^2x^-\mu_K^-(i)\left(f\left(x+\frac{i-1}{K}\delta_-\right)-f(x)\right)+\mathfrak{m} K^2 x^+(x^++x^-)\left(f\left(x-\frac{1}{K}\delta_+\right)-f(x)\right)\\&+\mathfrak{m} K^2 x^-(x^++x^-)\left(f\left(x-\frac{1}{K}\delta_-\right)-f(x)\right)+\theta^+Kx^+\left(f\left(x-\frac{1}{K}\delta_++\frac{1}{K}\delta_-\right)-f(x)\right)\\&+\theta^-Kx^-\left(f\left(x-\frac{1}{K}\delta_-+\frac{1}{K}\delta_+\right)-f(x)\right),
        \end{split}
\end{equation}
where we write $x = (x^+, x^-)$ and we use the notation $\delta_+ = (1,0)$ and $\delta_- = (0,1)$.
In particular, we obtain for $f:\mathbb{R}^2\rightarrow\mathbb{R}$ bounded, thrice continuously differentiable with vanishing third derivative (i.e.~quadratic functions), that
\[
\begin{split}
    A_K f(x) =& \sum_{i\geq 0} x^+\mu_K^+(i)\left( K(i-1)\partial_+ f(x) + \dfrac{(i-1)^2}{2}\partial_{++}f(x)\right)\\
    & + \sum_{i\geq 0} x^-\mu_K^-(i)\left( K(i-1)\partial_- f(x) + \dfrac{(i-1)^2}{2}\partial_{--}f(x)\right)\\
    &+\mathfrak{m}x^+(x^+ + x^-)\left( -K\partial_+ f(x) + \frac{1}{2}\partial_{++}f(x)\right)\\
    &+\mathfrak{m}x^-(x^+ + x^-)\left( -K\partial_- f(x) + \frac{1}{2}\partial_{--}f(x) \right)\\
    &+ \big(\theta^+x^+ - \theta^- x^-\big)\left( \partial_- f(x) - \partial_+ f(x) + \mathcal{O}\left(\dfrac{1}{K}\right)\right)\\
    =& \left( \mathfrak{s}^+x^+ + \theta^-x^- - \theta^+x^+\right)\partial_+ f(x)+ \left(\mathfrak{s}^-x^- + \theta^+x^+ - \theta^-x^-\right)\partial_- f(x)\\
    & + \dfrac{1}2\Big(x^+\big(\mathfrak{v}^+ + \mathfrak{m}(x^+ + x^-)\big)\partial_{+}^2f(x) + x^-\big(\mathfrak{v}^- + \mathfrak{m}(x^+ + x^-)\big)\partial_{-}^2 f(x)\Big)\\
    & + K\mathfrak{m}(1-(x^+ + x^-))\big( x^+, \, x^-\big)\cdot \nabla f(x) + \mathcal{O}\left(\dfrac{1 + x^+ + x^-}{K}\right).
\end{split}
\]
To simplify the notation, we write $\mathfrak{s}\odot x$ for entry-wise multiplication, extending it to the symbolic multiplication $\nabla\odot\nabla$, $\langle\cdot,\cdot\rangle$ for the usual scalar product, as well as $\Vert\vert x\vert\Vert := x^+ + x^-$ and $\overset{\pm}{\theta} := \binom{+\theta^+}{-\theta^-}$.
This allows us to rewrite the generator as
\[
A_Kf(x) = \Big(B_K f(x) + K F(x)^T\cdot\nabla f(x)\Big)\mathds{1}_{x\in U},
\]
for $U := \{ x\in \mathbb{R}^2\;:\; \Vert\vert x\vert\Vert > 0\}$, with an operator $B_K$ that satisfies, with $f$ as above, the convergence $B_K f \to Bf$ locally uniformly to the diffusion generator
\[
\begin{split}
Bf(x) =& \left\langle \mathfrak{s}\odot x + \left\langle\overset{\pm}\theta,x\right\rangle\binom{-1}{1} ,\nabla f(x)\right\rangle + \dfrac{1}{2}\left( \mathfrak{m}\Vert\vert x\vert\Vert^2\Delta f(x) + \left\langle \mathfrak{v}\odot x, (\nabla \odot\nabla)f(x)\right\rangle\right),
\end{split}
\]
and the vector field
\[
F:U\rightarrow\mathbb{R}^2,\quad x\mapsto \mathfrak{m}\big(1-\Vert\vert x\vert\Vert\big)x
\]
that pushes the process onto the manifold 
\[
\Gamma := F^{-1}(0)= \{(1-w, w)\;:\; w\in\mathbb{R}\}.
\]
The function
\[
\Phi: U\rightarrow\Gamma, \quad x\mapsto \dfrac{x}{\Vert\vert x\vert\Vert},
\]
maps any initial point $x\in U$ to its image on $\Gamma$ under the flow induced by $F$.

    For $k > 0$, define the set
    \[
    U_k := \left\{ x\in \mathbb{R}^2\;:\; \Vert\vert x\vert\Vert \in \left(\frac{1}{k},k\right), x^\pm \in (-k, k)\right\},
    \]
    which is relatively compact in $U$.
    Let
    \[
    \tau_K^k := \inf\left\{t\geq 0\;:\; X_K(t-)\not\in U_k\text{ or } X_K(t)\not\in U_k\right\},
    \]
    be the corresponding exit time.
\begin{lemma}\label{lem:marta_conditions_katzenberger_method}
    The quadratic variation $[X_K](t\wedge\tau_K^k)$ is uniformly integrable for all $t\geq 0$ and $k > 0$.
    Furthermore,
    \[
    \limsup_{K\to\infty} \mathbb{E}\left[\sup_{0\leq t\leq T\wedge\tau_K^k} \vert \Delta X_K(t)\vert^2\right] = 0,
    \]
    for all $T\geq 0$ and $k > 0$.
\end{lemma}
\begin{proof}
Since $X_K$ is a pure-jump Markov process,
\[
[X_K](t) = \sum_{0\leq s\leq t} \vert X_K(s)\vert^2.
\]
Clearly, for $t\leq \tau_K^k$, this quantity can be coupled with an independent family of processes $(N_i)_{i\geq 2}$ such that
\[
[X_K](t\wedge\tau_K^k) \leq \dfrac{1}{K^2}\sum_{i\geq 2} (i-1)^2 N_i(t),
\]
where $N_i$ is a Poisson process with rate
\[
\begin{cases}
    K^2\Big(k\big(\mu_K^+(0) + \mu_K^+(2) + \mu_K^-(0) + \mu_K^-(2)\big) + mk^2\Big) & \text{ for }i=2,\\
    K^2k\big(\mu_K^+(3) + \mu_K^-(3)\big) + Kk(\theta^+ + \theta^-) & \text{ for }i=3,\\
    K^2k\big(\mu_K^+(i) + \mu_K^-(i)\big) & \text{ for }i\geq 4.
\end{cases}
\]
In particular, since the processes $(N_i)_{i\geq 2}$ are independent, we may write $[X_K](t\wedge\tau_K^k)\leq P_K(t)/K^2$, where $P_K$ is in law a Poisson random variable with parameter
\[
\lambda_K^k := K^2k\sum_{k\geq 0} (i-1)^2 \big( \mu_K^+(i) + \mu_K^-(i)\big) + K^2 mk^2 + Kk(\theta^+ + \theta^-),
\]
which has second moment $(\lambda_K^k)^2 + \lambda_K^k = \mathcal{O}(K^4)$, so that
\[
\sup_{K\in\mathbb{N}} \mathbb{E}\left[ [X_K](t\wedge\tau_K^k)^2\right] < +\infty.
\]
This implies in particular that it is uniformly integrable.
Finally, for $i\geq K/\log K$, we have
\begin{align*}
\mathbb{P}\left(\sup_{0\leq t\leq T\wedge\tau_K^k} \vert \Delta X_K(t)\vert^2 = \dfrac{(i-1)^2}{K^2}\right) &\leq \mathbb{P}(N_i(T) \geq 1) \\
&\leq \mathbb{E}[N_i(T)]\\
&= K^2Tk\big(\mu_K^+(i) + \mu_K^+(i)\big),
\end{align*}
so that
\begin{align*}
    \mathbb{E}\left[\sup_{0\leq t\leq T\wedge\tau_K^k} \vert \Delta X_K\vert^2\right] &\leq \mathcal{O}\left(\dfrac{1}{K^2}\cdot\left(\dfrac{K}{\log K}\right)^2\right) + \sum_{i\geq \frac{K}{\log K}} \dfrac{(i-1)^2}{K^2}\cdot K^2Tk\big(\mu_K^+(i) + \mu_K^-(i)\big)\\
    &= \mathcal{O}\left(\dfrac{1}{\log^2 K}\right) + Tk\sum_{i\geq \frac{K}{\log K}} (i-1)^2\big(\mu_K^+(i) + \mu_K^-(i)\big).
\end{align*}
Let $\eta > 0$ small enough so that $\sum_{i\geq 0} i^{2 + \eta}\mu_K^\pm(i)$ is uniformly bounded in $K$, then
\begin{align*}
    \sum_{i\geq \frac{K}{\log K}} (i-1)^2\mu_K^\pm(i) &\;\leq \left(\dfrac{\log K}{K}\right)^\eta\sum_{i\geq \frac{K}{\log K}} i^{2+\eta} \mu_K^\pm(i) \to 0,
\end{align*}
as $K\to \infty$.
\end{proof}

\begin{proof}[Proof of \Cref{theo:convergence_at_carrying_capacity}]
Since the idea is to apply  \Cref{theo:katzenberger_method}, we start by verifying \Cref{ass:katz_regularity_deter}. We note that the deterministic system induced by the vector field $F$ is well-behaved.
Indeed, all three $F$, $\Phi$ and $\Gamma$ are smooth.
Furthermore, we compute for $x = (w, 1-w)\in\Gamma$ that
\[
    \partial F(x) = \mathfrak{m}\begin{pmatrix}
        -(1-w) & -w\\ -(1-w) & -w
    \end{pmatrix},
\]
which has eigenvalues $0$ and $-\mathfrak{m} < 0$.
Since $\Gamma$ is of codimension $1$, the manifold is asymptotically stable for the flow induced by $F$.

Now, as every compact in $U$ is contained in some $U_k$, $k > 0$, \Cref{lem:marta_conditions_katzenberger_method} allows us to invoke \Cref{theo:katzenberger_method}.
We conclude that $(X_K^{\tau_K^2})_K$ is tight and that any limit point $X$ takes values in $\Gamma$ almost surely.
Furthermore, since $X_K^\pm(t) \geq 0$ for all $t\geq 0$, we conclude that $X^\pm(t) \geq 0$ a.s.\ for all $t\geq 0$.
In particular, $0\leq X^\pm \leq 1$ almost surely.
Writing $\tau^2$ for the exit time of $X$ from the set $U_2$, this implies $\tau^2 = +\infty$ almost surely and \Cref{theo:katzenberger_method} lets us conclude that $X$ is a solution to the martingale problem for the operator
\[
\begin{split}
    C f(x) &= \left( \partial\Phi(x)b(x) + \dfrac{1}{2}\sum_{i,j\in\pm} \partial_{ij}\Phi(x)c_{ij}(x)\right)^T\cdot\nabla f(x) \\
    &\qquad + \dfrac{1}{2}\sum_{i,j,k,\ell\in\pm}^d \partial_k\Phi_i(x)\partial_\ell\Phi_j c_{k\ell}(x)\partial_{ij}f(x),
\end{split}
\]
where, for $x = (1-w,w)\in\Gamma$,
\[
b(x) = \binom{\mathfrak{s}^+(1-w) + \theta^-w - \theta^+(1-w)}{\mathfrak{s}^-w + \theta^+(1-w) - \theta^- w},
\]
and $c_{11} = (1-w)(\mathfrak{m}+\mathfrak{v}^+)$, $c_{22} = w(\mathfrak{m} + \mathfrak{v}^-)$ and $c_{12} = 0 = c_{21}$.
Now, for $x = (1-w, w)\in \Gamma$, we have
\[
\partial\Phi(x) = \begin{pmatrix}
    w & -(1-w)\\ -w & 1-w
\end{pmatrix},
\]
as well as
\[
\partial_{+}^2\Phi(x) = 2\binom{-w}{w}\quad\text{ and }\quad \partial_{-}^2\Phi(x) = 2\binom{1-w}{-(1-w)}.
\]
Using the fact that $X$ is of the form $(1-W, W)$, we may conclude that $W$ solves the martingale problem for
\begin{align*}
    \widehat{C}f(w) &= \Big(-w\big(\mathfrak{s}^+(1-w) + \theta^-w - \theta^+(1-w)\big) \\
    &\qquad\quad + (1-w)\big(\mathfrak{s}^-w + \theta^+(1-w) - \theta^- w\big)\\
    &\qquad\quad + (1-w)(\mathfrak{m}+\mathfrak{v}^+)w - w(\mathfrak{m} + \mathfrak{v}^-)(1- w)\Big) f'(w)\\
    &\quad + \Big(w^2(\mathfrak{m} + \mathfrak{v}^+)(1-w) + (1-w)^2(\mathfrak{m} + \mathfrak{v}^-)w\Big)f''(w)\\
    &= \Big( -\big( (\mathfrak{s}^+ - \mathfrak{v}^+) - (\mathfrak{s}^- - \mathfrak{v}^-)\big)w(1-w) + \theta^+(1-w) - \theta^-w\Big)f'(w)\\
    &\quad + \Big( \mathfrak{m} + \mathfrak{v}^+w + \mathfrak{v}^-(1-w)\Big)w(1-w)f''(w).
\end{align*}
Strong uniqueness of \eqref{eq:Gillespie-WF} follows by adapting \cite[Theorem 2]{CPP21}. Using the fact that having a solution to this martingale problem corresponds to the existence of a weak solutions to \eqref{eq:Gillespie-WF}, see e.g.~\cite{Kurtz10}, we conclude that the martingale problem for $\widehat{C}$ is well-posed and that the solution is the unique strong solution to \eqref{eq:Gillespie-WF}.
The remaining claim on duality follows from \cite[Theorem 2]{CPP21}.
\end{proof}

\section{The genealogy at carrying capacity}\label{sec:limit-genealogy}

Classically, the duality relation \eqref{eq:moment-duality} has been linked to a sampling duality between the forward-in-time dynamics of a finite population and the lineage counting process of the corresponding genealogy.
In the presence of selective forces, this interpretation comes with a caveat: instead of being linked to the actual genealogy of an untyped sample, the duality is based on the \emph{potential ancestry} which allows for branching when a reproduction may come from a selected type. This leads to the Ancestral Selection Graph (ASG) introduced in \cite{KN97}.
In the setting of fluctuating population sizes, this analysis becomes more involved as the backward-in-time dynamics of the genealogy depends on the current population size and usually cannot be described as an autonomous Markov process.

The goal of this section is to provide this genealogical interpretation of the duality relation \eqref{eq:moment-duality} through the construction of an ASG whose lineage counting process converges to the dual process \eqref{eq:dual-generator}.
Through the encoding of the ASG as metric measure spaces, this analysis can be extended to study the convergence of the potential ancestry itself to a limiting object, see e.g.~the introduction of \cite{felix24} and the references therein for an overview.
Although we expect the arguments in this section to adapt to that setting, the extension to genealogical metric measure spaces is beyond the scope of this paper and will be left to future work.

We will proceed as follows: we first introduce, in \Cref{sec:underlying_assumptions}, the assumptions we need for the construction of the ASG; in \Cref{sec:graph-construction-forward}, we give a graphical construction of the population size process $\left(N_K^+,N_K^-\right)$, and then define, in \Cref{sec:ASG}, the corresponding ASG; finally, in \Cref{sec:convergence-genealogy}, we will prove the convergence of the lineage counting process through a coupling argument.

As the proof itself is rather involved, we have decided to omit the mutation dynamics completely.
Note, however, that mutations from the selected to the weak type can be easily recovered \emph{a posteriori} by considering an appropriate Poisson process on the lineages marking the mutation events.
At one such event, the affected lineage is then pruned in the ASG as the individual before the mutation had to be of the selected type.
For mutations from the weak to the selected type, the situation is more complicated, see also \cite{BW17} in which the authors introduce the \emph{killed ASG} to accommodate for this type of mutation.

\begin{remark}
Since we build the ASG for a population that has already reached the carrying capacity, we always consider time to be rescaled by a factor $K$ as in \Cref{theo:convergence_at_carrying_capacity}. 
In the following, we will therefore slightly abuse notation and write $N_K(t)$ instead of $N_K(Kt)$.
\end{remark}

\subsection{Underlying assumptions}\label{sec:underlying_assumptions}

The whole construction relies on two sets of assumptions, the first of which will ensure that we indeed are in the diffusive regime.
More precisely, we will suppose the assumptions from \Cref{theo:convergence_at_carrying_capacity} to hold in order for the rescaled population size to converge to $1$.
Additionally, we assume the moment bounds
\begin{equation}\label{eq:higher-moments-bound}
\sum_{i\geq 0} i^\ell \mu_K^\pm(i) = \text{o}\left(K^{\ell - 2}\right)
\end{equation}
to hold for $\ell\geq 3$.
Note that the bounding constant may depend on $\ell$.
This weak moment assumption ensures that no more than two lineages will coalesce at once and can be thought of extending the assumption
\[
\sum_{i\geq 0} i^{2+\eta} \mu_K^\pm(i)  < +\infty,
\]
for some $\eta > 0$ (although the latter does not follow from \eqref{eq:higher-moments-bound}).

Less apparent is that we need the two reproduction mechanisms $\mu^-$ and $\mu^+$ to be \emph{comparable}.
More precisely, to construct the ASG, we will couple the two mechanisms in such a way that deaths and reproductions occur through common \emph{events} at which an individual is chosen at random from the population and reproduces depending on its type.
The selective advantage stems from the fact that individuals of the selected type bear \emph{more} offspring at each reproduction event and participate \emph{less} often in death events.
To separate reproduction from death, set $\widehat{\nu}_K^\pm := \mu_K^\pm(\cdot \cap \mathbb{N})$.
Then, throughout this section, we assume that
\begin{equation}\label{eq:ordering_measures}
\widehat{\nu}_K^-([i, \infty)) \leq \widehat{\nu}_K^+([i, +\infty)),
\end{equation}
for all $i\in\mathbb{N}$ as well as
\begin{equation}\label{eq:ordering_death_rates}
    \mu^+_K(0) \leq \mu^-_K(0).
\end{equation}
Since reproduction with one offspring does not influence the dynamics of the process, we may assume that $\mu_K^+$ and $\mu_K^-$ have the same total mass (by adding a suitable multiple of $\delta_1$ to the one with less mass).
Following \cite{CKP23}, \eqref{eq:ordering_measures} guarantees the existence of a \emph{coupling measure} $\widehat{\nu}_K$ on $\mathbb{N}\times \mathbb{N}_0$ satisfying
\[
\widehat{\nu}_K^-(A) = \widehat{\nu}_K\big(\{(i, j)\;:\; y\in A, j\in\mathbb{N}_0\}\big) \qquad\text{ and }\qquad \widehat{\nu}_K^+(B) = \widehat{\nu}_K\big(\{(i,j)\;:\; i+j\in B\}\big),
\]
for all $A,B\subseteq \mathbb{N}$.
An unfortunate side effect of this assumption is that we can only treat the case $\mathfrak{v}^+ = \mathfrak{v}^- =: \mathfrak{v}$, see \Cref{lem:properties_mixed_moments}.

Before going into the rigorous definition of the ASG (which we provide in \Cref{sec:ASG}), let us give an intuition on how death and reproduction influence the potential ancestry.
We first consider a reproduction event $(i,j)$, meaning that a parent is chosen uniformly from the population and is replaced by either $i$ or $i+j$ children, depending on its type.
More precisely, there are $i$ children if a wild type reproduces and $i+j$ children if the selected type reproduces.
If we think of the individuals in the population labeled from $1$ to $n\in\mathbb{N}$, this is equivalent to shuffling the labels uniformly and picking the individual with label $1$ as parent. The children will take labels $\{1,\dots,i\}$ if the parent is of the wild type, or $\{1,\dots,i+j\}$ if the parent is selected; the labels of the remaining population being shifted upwards accordingly, see \Cref{fig:point-processes}.
Then, we colour the $i$ lowest labels in blue and the labels $i+1,\dots, i+j$ in red.

Backwards in time, if a lineage falls onto a blue label, it surely has participated in the reproduction event and will merge into the parent individual with label $1$.
If, on the other hand, the lineage falls onto a red label, we do not know whether it has participated in the reproduction event.
The lineage will therefore merge into the one with label $1$ and also \emph{branch}.
Finally, if a lineage does not fall on a coloured label, it simply follows the change in label.
Note that this may lead to additional mergers that vanish asymptotically (see \Cref{sec:ASG}).

For the deaths, we will also need a consistent coupling.
From the fact that $\mu_K^+(0) \leq \mu_K^-(0)$, we may first throw death events at \emph{per capita} rate $\mu_K^+(0)$ at which a individual will be chosen uniformly from the population and \emph{independently from its type} to die.
These events will not affect the potential ancestry.
Additionally, we will throw selective death events at \emph{per capita} rate $\mu_K^-(0) - \mu_K^+(0)$ at which we select an individual uniformly at random which only dies if of wild type.
As before, we may reformulate these additional death events by reshuffling the labels and killing the individual with label $1$, if it is of wild type.
All other individuals will decrease their label by one accordingly.
Backwards in time, if a lineage falls onto label $1$ and encounters a selective death event, we again have to follow both potential ancestors so that the lineage branches into the lineages with label $1$ and $2$.

Comparing the two different mechanisms, we see that we can interpret the two types of death events as reproduction events of the form $(0,0)$ or $(0,1)$, depending on whether it is common or selective. 
In the following, we will therefore consider the modified ``reproduction" measure 
\[
\nu_K = \widehat{\nu}_K + \mu_K^+(0)\delta_{(0,0)} + (\mu_K^-(0) - \mu_K^+(0))\delta_{(0,1)}
\]
on $\mathbb{N}_0\times \mathbb{N}_0$, with $\nu_K(0,j) = 0$ for all $j>1$, satisfying the moment bounds \eqref{eq:higher-moments-bound} as well as
\begin{equation}\label{eq:average_selective_advantage}
\sum_{i,j\geq 0} j \nu_K(i,j) = \mathfrak{m}_K^+ - \mathfrak{m}_K^- = \dfrac{\mathfrak{s}^+-\mathfrak{s}^-}{K} + \text{o}\left(\frac{1}{K}\right),
\end{equation}
and
\begin{equation}\label{eq:variance_of_non_selective_stuff}
\sum_{i,j\geq 0} i(i-1)\nu_K(i,j) = \sum_{i,j\geq 0} \Big( (i-1)^2 + (i-1)\Big)\nu_K(i,j) = \mathfrak{v}^-_K + \mathfrak{m}_K^-.
\end{equation}
For further reference, set $\mathfrak{s} := \mathfrak{s}^+ - \mathfrak{s}^-\geq 0$.
Finally, note that by rescaling time by a constant factor, we may and will assume that $\nu_K$ is a probability measure.

\begin{lemma}\label{lem:properties_mixed_moments}
We have that 
\[
\lim_{K\to \infty}\sum_{i,j\geq 0} j^2 {\nu}_K(i,j) = 0 \qquad\text{ and }\qquad 
\lim_{K\to \infty} \sum_{i,j\geq 0} ij{\nu}_K(i,j) = 0.
\]
It follows in particular that $\mathfrak{v}^+ = \mathfrak{v}^-$.
\end{lemma}
\begin{proof}
The last implication follows through the observation
\[
\mathfrak{v}^+_K - \mathfrak{v}^-_K = \sum_{i\geq 0} (i - 1)^2 \big( \mu_K^+(i) - \mu_K^-(i)\big) = \sum_{i,j\geq 0} \Big((i+j-1)^2 - (i-1)^2\Big) {\nu}_K(i,j),
\]
together with
\begin{align*}
\sum_{i,j\geq 0} \Big((i+j-1)^2 - (i-1)^2\Big) {\nu}_K(i,j) = \sum_{i,j\geq 0} \Big(j^2 + 2(i-1)j\Big) {\nu}_K(i,j) \underset{K\to \infty}{\longrightarrow} 0,
\end{align*}
from the other claims.
Next, we use Hölder's inequality, with $p=q=2$, to obtain the bound
\begin{align*}
\sum_{i,j\geq 0} j^2 {\nu}_K(i,j)\leq\sqrt{\sum_{i,j\geq 0} j\nu_K(i,j)}\sqrt{\sum_{i,j\geq 0} j^3{\nu}_K(i,j)} &= \mathcal{O}\left(K^{-1/2}\right)\cdot \text{o}\left(K^{1/2}\right) = \text{o}(1),
\end{align*}
where we used that 
\[
\sum_{i,j\geq 0} j^3{\nu}_K(i,j) \leq \sum_{i,j\geq 0} (i+j)^3 {\nu}_K(i,j) \leq \mu_K^-(0) - \mu_K^+(0) + \sum_{i\geq 0} i^3 \mu_K^+(i) = \text{o}(K).
\]
The second claim follows analogously.
\end{proof}

Recall that the population size process $\left(N_K(t)/K\right)_{t\geq 0}$ converges weakly in $\mathbb{D}([0,+\infty); \mathbb{R})$ to the constant function equal to $1$.
Since the limit is deterministic and continuous, the convergence automatically holds in probability and w.r.t.~the topology of uniform convergence on compacts.
In particular, we may choose a sequence $(\epsilon_K)_{K\in\mathbb{N}}$ of positive numbers with $\epsilon_K \to 0$, as $K\to \infty$, and such that
\begin{equation}\label{eq:varepsilon-condition}
    \lim_{K\to \infty}\P\left(\frac{N_K(t)}{K}\in[1-\varepsilon_K,1+\varepsilon_K], \,\forall t \in [0,T]\right)=1.
\end{equation}
Defining $N_K^\downarrow := \lfloor (1 - \epsilon_K)K\rfloor$ and $N_K^\uparrow := \lceil (1 + \epsilon_K)K\rceil$, the above ensures that the event 
\[
\mathcal{E}_K := \left\{ N_K(t) \in \left[ N_K^\downarrow, N_K^\uparrow\right], \,\forall t\in [0,T]\right\},
\]
occurs asymptotically almost surely.
As such, it suffices to consider all derived processes only on this event.
We will make use of this fact not only for the analysis itself, but also to brush over some technicalities in the construction by leaving processes undefined (or better: arbitrarily defined) outside of $\mathcal{E}_K$.

\subsection{Graphical construction of $(N_K^+, N_K^-)$}\label{sec:graph-construction-forward}

In the following, we construct the population size process $\left(N_K^+(t),N_K^-(t)\right)_{t\in [0,T]}$, where $T>0$ is a finite time horizon, from a Poisson point process that will allow us to define the ASG explicitly.
The reproduction events will be driven exclusively by a Poisson point process $\xi_K$ on $[0,T]\times \mathbb{N}_0\times \mathbb N_0$ with mean intensity $K\mathrm{d}t\otimes N_K^\uparrow\nu_K(\mathrm{d}i, \mathrm{d}j)$.
One should think of $\xi_K$ as containing all \emph{possible} reproduction events.
The decision to \emph{effectively} use a reproduction event will be made based on the value of a uniform random variable $U(t,i,j)$ on $[0,1]$, sampled independently for every $(t,i,j)\in\xi_K$.

\begin{figure}
    \centering
    \includegraphics[width=0.8\linewidth]{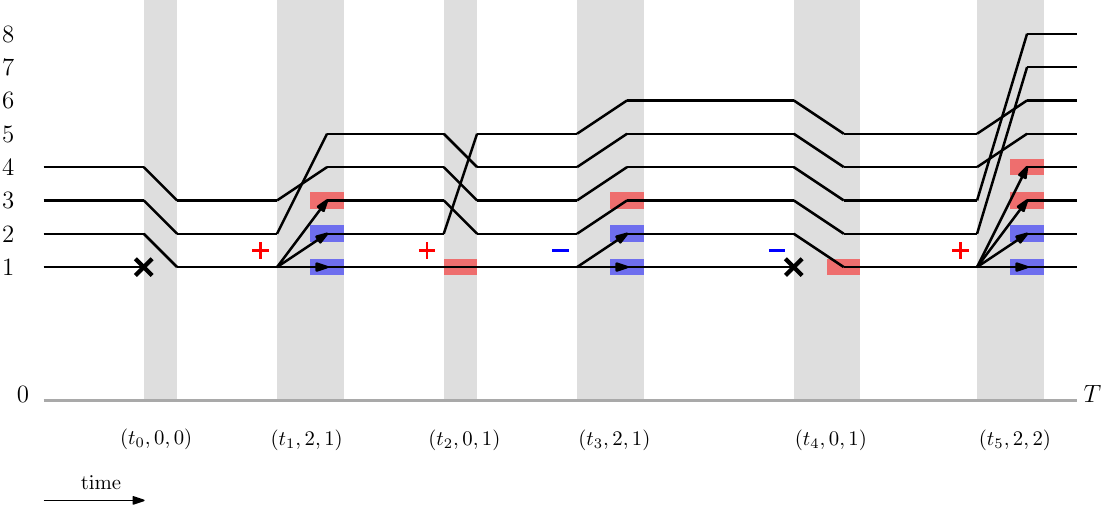}
    \caption{We show a realisation of the graphical representation of the population. Time increases from the left to the right; we show, at the bottom, the points of $\hat{\xi}_K$ and, on the left, all the possible labels. We do not show the shuffling of the lineages and only highlight the type of the individual with label $1$ before every reproduction event.}
    \label{fig:point-processes}
\end{figure}

Next, we construct a second point process $\hat{\xi}_K$ and a graphical representation of the population of size $\left(N_K^+(t),N_K^-(t)\right)_{t\in [0,T]}$ simultaneously as follows.
Set $N_K(0)^\pm$ to the initial condition of the population and \mbox{$\hat\xi_K \cap \left\{ \{0\}\times \mathbb{N} \times \mathbb{N}_0\right\} := \emptyset$}.
The initial individuals carry labels going from $1$ to $N_K(0)$.
Suppose that both the process $\left(N_K^+,N_K^-\right)$ and the point process $\hat\xi_K$ have been constructed up to some time $t_0$.
Recall that we define $N_K (t_0)= N_K^+(t_0) + N_K^-(t_0)$.
Let $(t,i,j)\in\xi_K$ be the first event after time $t_0$.
At the event, two things may happen: if $U(t,i,j) \geq \frac{N_K(t_0)}{N_K^\uparrow}$, nothing happens; otherwise the following procedure starts:
\begin{enumerate}
    \item The point $(t,i,j)$ is added to $\hat{\xi}_K$.
    
    \item The labels of the individuals are shuffled uniformly and independently from the rest, then the individuals with label $1$ is selected as parent.

    \item If the parent is of type $-$, it produces $i$ children with labels $1, \dots, i$ and the labels of all other individuals are shifted up (or down, if $i=0$) by $i-1$; if the parent is of type $+$, it produces $i+j$ children with labels $1, \dots, i+j$ and the labels of the other individuals change as follows: the individuals with label $\ell > 1 + j$ shift their label up by $i-1$, the remaining $j$ individuals are shifted to the topmost $j$ labels.
    The population size changes accordingly by $(i-1, 0)$ or by $(0, i+j-1)$.

    \item The lowest $i$ labels are coloured in blue, the labels $i+1,\dots, i+j$ are coloured in red. If $i = 0 = j$, no label is coloured. 
\end{enumerate}
If the population size leaves the interval $[N_K^\downarrow, N_K^\uparrow]$ during an event, the process is stopped and does not change anymore.
See Figure \ref{fig:point-processes} for a sketch of the above construction.

\subsection{Construction of the ASG}\label{sec:ASG}

Consider a uniform sample of $m$ labels from the population alive at time $T$. 
The ASG traces the potential ancestry of the associated individuals backward in time without using the \emph{type} of the individuals.
This leads to the following heuristic definition: the ASG jumps at most at events $(T-t, i,j)\in\hat\xi_K$ and follows the evolution of the labels with one additional rule: any lineage that falls onto a red label $i+k$ will instead branch onto both the individuals with labels $1$ and $1+k$.
See Figure \ref{fig:real-ancestral-selection} for a representation of the ASG. 

We highlight that the ASG is \emph{not} Markovian in the sense that, at an event $(T-t, i,j)$, it does not only make use of its past and the information $(N_{T-s})_{s\in[0,T]}$, but also of the type of the parent individuals or, equivalently, $N_{(T-t)-}$.
Indeed, every time a lineage falls onto one of the $j$ top-most labels, it is impossible to predict which route the ancestral line will take without the knowledge of the parent type.
Note, however, that this problem almost never affects the \emph{topology} of the potential ancestry, in the sense that it affects whether ancestral lines merge or branch only in the following two scenario: at an event $(i,j)$, at least one lineage falls onto one of the $j$ top-most labels and at least one additional lineage falls onto a red spot. We show later in this section that this scenario has vanishing probability.
In all other cases, it might be unclear where exactly a lineage moves, but the number of lineages merging and branching will be known.
Together with the fact that the labels are shuffled at every event, this will allow us to define a Markovian \emph{lineage counting process}.

More precisely, write $A_K(t)$ for the number of lineages in the ASG alive at time $t\in[0,T]$.
Again, $A_K$ jumps at most at points in $\hat\xi_K$. 
Furthermore, since the lineages are shuffled uniformly at every event, we can compute the transitions of $A_K$ explicitly with the following caveat: in the problematic case of having one lineage in the $j$ top-most lineages and one on a red label, we will send the lineage counting process to a cemetery state $\ast$.
More precisely, setting $\hat\xi_K(t) := \hat\xi_K\cap [T-t,T]\times \mathbb N_0^2$, we have that $t\mapsto (\hat\xi_K(t), N_K(T-t), A_K(t))$ is a pure jump Markov process w.r.t.~its natural filtration and at every $(T-t,i,j ) \in \hat{\xi}_K$, $A_K$ jumps according to the following transition probabilities:
\begin{equation}\label{eq:transition_prob_A_K}
n\mapsto\begin{cases}
    n + 1& \text{ with probability } p_{i,j}^+(n, N_K(T-t)) + \widehat{p_{i,j}}^+(n,N_K(T-t)),\\
    n - 1 & \text{ with probability } p_{i,j}^-(n, N_K(T-t)) + \widehat{p_{i,j}}^-(n, N_K(T-t)),\\
    \ast & \text{ with probability } p_{i,j}^\ast(n, N_K(T-t)),\\
    n - k + 1 & \text{ for $3\leq k\leq n$ with probability } p_{i,j}^k( n, N_K(T-t)),\\
    n & \text{ otherwise }
\end{cases}
\end{equation}
where, in the context of drawing $n$ lineages uniformly at random from $N$ possible spots ($i$ of which are blue and $j$ are red),
\begin{itemize}
    \item the expression
    \[
        p_{i,j}^+(n, N) = \dfrac{\binom{j}{1}\binom{N - (i+2j)}{n-2}}{\binom{N}{n}} = n\cdot\dfrac{j}{N}\cdot\prod_{k=1}^{n-1} \left(1 - \dfrac{(i+j-1) + j}{N-k}\right)
    \]
    denotes the probability of having exactly one lineage hit a red label and none hit a blue label or one of the $j$ top-most labels;

    \item the expression
    \[
        \widehat{p_{i,j}}^+(n,N) = \sum_{\ell = 2}^{n\wedge(i+j)} \dfrac{\binom{j}{\ell}\binom{N-(i+2j)}{n-\ell}}{\binom{N}{n}}
    \]
    denotes the probability of having at least two lineages hit a red label and none hit a blue label or one of the $j$ top-most labels;

    \item the expression
    \[
        p_{i,j}^-(n,N) = \dfrac{\binom{i}{2}\binom{N-(i+j)}{n-2}}{\binom{N}{n}} = \binom{n}{2}\cdot \dfrac{i(i-1)}{N(N-1)}\cdot \prod_{k=2}^{n-1}\left( 1 - \dfrac{(i-1) + (j-1)}{N-k}\right)
    \]
    denotes the probability of having exactly two lineages fall onto a blue label and none fall onto a red label;
    
    \item the expression
    \[
    \widehat{p_{i,j}}^-(n, N) = \sum_{\ell=1}^{(n-2)\wedge j} \dfrac{\binom{i}{2}\binom{j}{\ell}\binom{N - (i+2j)}{n-2-\ell}}{\binom{N}{n}}
    \]
    denotes the probability of having exactly two lineages fall onto a blue label, at least one fall onto a red label and none hit one of the $j$ top-most labels;

    \item the expression $p_{i,j}^\ast(n,N)$, with
    \[
        p_{i,j}^\ast(n,N) \leq \dfrac{\binom{j}{1}^2\binom{N}{n-2}}{\binom{N}{n}} = n(n-1)\cdot \dfrac{j^2}{(N-n+1)(N-n+2)},
    \]
    denotes the probability of having at least one lineage fall onto a red label and at least one lineage be among the $j$ top-most lineages;

    \item the expression $p_{i,j}^k(n,N)$, with
    \[
        p_{i,j}^k(n,N) \leq \dfrac{\binom{i}{k}\binom{N - i}{n-k}}{\binom{N}{n}},
    \]
    denotes the probability of having exactly $k$ lineages fall onto a blue label, any number fall onto a red label, and none fall onto one of the $j$ top-most labels if at least one red label is hit.
    
\end{itemize}

The following two results on the behaviour of the different probabilities will be useful in the next section.

\begin{figure}
    \centering
    \includegraphics[width=0.8\linewidth]{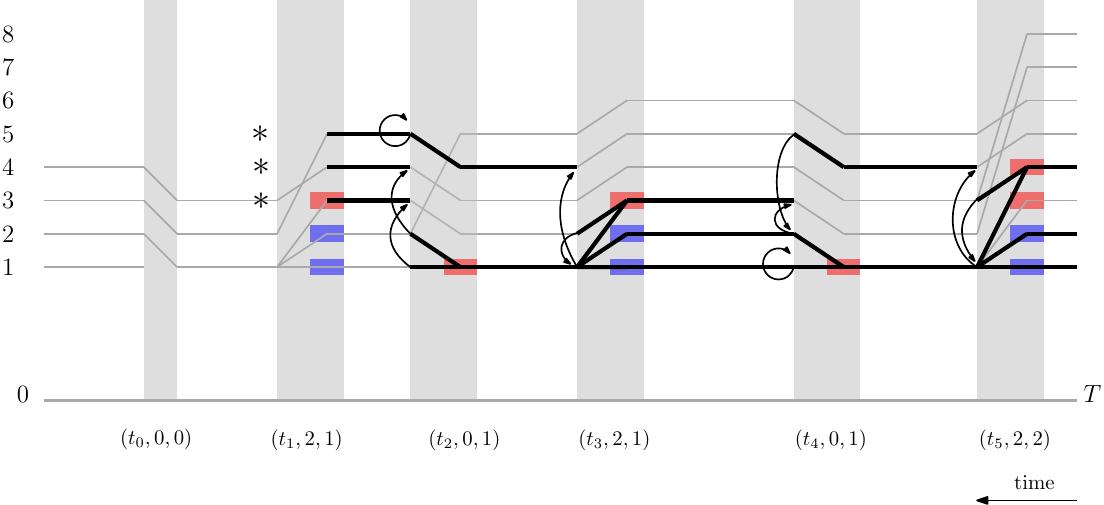}
    \caption{Given the graph of Figure \ref{fig:point-processes}, consider the genealogy of a sample of $m=3$ individuals alive at time $T$. At a reproduction event, we proceed with the coalescing and/or branching update. The arrows then show how we shuffle the remaining lineages after update. At the second to last event, we have a lineage on the red label and one on the top, hence the process goes in the cemetery state.}
    \label{fig:real-ancestral-selection}
\end{figure}

\begin{lemma}\label{lem:7_probabilities_ordered}
    Let $\square\in\{+,-\}$.
    For every $L\in\mathbb{N}$, there exist $K_L\in\mathbb{N}$ and $\kappa_L> 0$ such that for all $1\leq n\leq L$, $K \geq K_L$ and $0\leq i,j\leq \kappa_L K$, we have that $N\mapsto p_{i,j}^\square(n, N)$ is monotonically decreasing on $[N_K^\downarrow, N_K^\uparrow]$.
    Furthermore,
    \[
    \sup_{1\leq n\leq L} \left\vert p_{i,j}^\square(n,N_K^\uparrow) - p_{i,j}^\square(n, N_K^\downarrow)\right\vert \leq C_L\cdot \left(\dfrac{j}{K} + \dfrac{i(i-1)}{K^2}\right)\epsilon_K,
    \]
    for all $K\geq K_L$ and some constant $C_L$ only depending on $L$.
\end{lemma}
\begin{proof}
Since the computations are similar, we will concentrate on $p_{i,j}^+(n,\cdot)$.
    For $n = 1$, the assertion is trivial, so let $n\geq 2$.
    First note that $p_{i,j}^+(n,\cdot)$ is differentiable as a function $(0,+\infty) \to \mathbb{R}$ with derivative 
    \begin{align*}
    &\partial_N p^+_{i,j}(n,N) \\
    &= nj\left( \dfrac{-1}{N^2}\prod_{k=1}^{n-1}\left(1 - \dfrac{(i+j-1) + j}{N-k}\right) + \dfrac{1}{N}\sum_{k=1}^{n-1} \dfrac{(i+j - 1) + j}{(N-k)^2}\prod_{\substack{\ell=1\\\ell\neq k}}^{n-1}\left(1 - \dfrac{(i+j-1)+j}{N-\ell}\right)\right)\\
    &= \sum_{k=1}^{n-1} \underbrace{\dfrac{nj}{N}\left(\prod_{\substack{\ell = 1\\\ell\neq k}}^{n-1} \left( 1 - \dfrac{(i+j-1)+j}{N-\ell}\right)\right)}_{\geq 0}\left( \dfrac{(i+j-1)+j}{(N-k)^2} - \dfrac{1}{N(n-1)}\left(1 - \dfrac{(i+j-1+j)}{N-k}\right)\right).
\end{align*}
Next, if $i,j\leq \kappa K$ and $2\leq n\leq L$, we have that the second factor, for $N\in \left[N_K^\downarrow, N_K^\uparrow\right]$, is bounded from above by
\begin{align*}
    \dfrac{3\kappa K}{(N_K^\downarrow-L)^2} - \dfrac{1}{N_K^\uparrow L}\left(1 - \dfrac{3\kappa K}{N_K^\downarrow - L}\right) &\leq \dfrac{6\kappa}{K} - \dfrac{1}{2LK}\cdot (1 - 6\kappa)\\
    &\leq \dfrac{18L\kappa - 1}{2K} < 0
\end{align*}
whenever $\kappa < \frac{1}{18L}$ and $K > L/4$ large enough such that $\epsilon_K < \frac{1}{4}$.
This proves that $p_{i,j}^+(n,\cdot)$ is decreasing on $[N_K^\downarrow, N_K^\uparrow]$.

Again, the variation estimate follows immediately for $n = 1$ as
\[
\left\vert p_{i,j}^+(1, N_K^\uparrow) - p_{i,j}^+(1, N_K^\downarrow)\right\vert = j\left(\dfrac{1}{N_K^\downarrow} - \dfrac{1}{N_K^\uparrow}\right) \leq 4\cdot \dfrac{j}{K}\cdot \epsilon_K
\]
for $K$ large enough.
Next, for $2\leq n\leq L$ and $i,j\leq \kappa K$, we have that for $K$ large enough,
\begin{align}
    \vert \partial_N p_{i,j}^+(n, N)\vert &\leq \sum_{k=1}^{n-1} \dfrac{2nj}{K} \cdot\left(\dfrac{6\kappa K}{K^2} + \dfrac{4}{K(n-1)}\right) \leq C_L \cdot \dfrac{j}{K^2}.
\end{align}
This yields
\[
\left\vert p_{i,j}^+(n, N_K^\uparrow) - p_{i,j}^+(n, N_K^\downarrow)\right\vert \leq 2K\epsilon_K\cdot C_L\cdot \frac{j}{K^2} \leq C_L'\cdot \frac{j}{K}\cdot \epsilon_K.
\]
\end{proof}

\begin{lemma}\label{lem:7_convergence_of_probabilities}
    Fix $L\in\mathbb{N}$.
    Then,
    \[
    \lim_{K\to \infty}\sup_{\substack{1\leq n\leq L\\N\in[N_K^\downarrow,N_K^\uparrow]}} \left\vert n\mathfrak{s} - K^2\sum_{i,j\in\mathbb N_0} p_{i,j}^+(n, N)\nu_K(i,j)\right\vert = 0
    \]
    and
    \[
    \lim_{K\to \infty} \sup_{\substack{1\leq n\leq L\\N\in[N_K^\downarrow,N_K^\uparrow]}} \left\vert \binom{n}{2}(\mathfrak{v} + \mathfrak{m}) - K^2\sum_{i,j\in\mathbb N_0} p_{i,j}^-(n, N)\nu_K(i,j)\right\vert = 0.
    \]
    Furthermore, let 
    \[
    \overline{p}_{i,j}(n, N) := p_{i,j}^\ast(n,N) + \widehat{p_{i,j}}^+(n, N) +\widehat{p_{i,j}}^-(n, N) + \sum_{k=3}^n p_{i,j}^k(n, N),
    \]
    then,
    \[
    \lim_{K\to \infty} \sup_{\substack{1\leq n\leq L\\N\in [N_K^\downarrow, N_K^\uparrow]}} K^2\sum_{i,j\in\mathbb N_0} \overline{p}_{i,j}(n,N)\nu_K(i,j)  = 0.
    \]
\end{lemma}
\begin{proof}
First, note that
\[
K^2\nu_K(i+j\geq \kappa_L) \leq \dfrac{K^2}{(\kappa_LK)^{2+\eta}}\sum_{i,j\geq \kappa_L K} (i+j)^{2+\eta}\nu_K(i,j) = \mathcal{O}(K^{-\eta}),
\]
so that we may apply the previous lemma, which in combination with \eqref{eq:average_selective_advantage} and \eqref{eq:variance_of_non_selective_stuff} allows us to reduce the problem to $N = K$.
Then, we check that
\begin{align*}
    &K^2\sum_{i,j\in\mathbb{N}_0} p_{i,j}^+(n, K)\nu_K(i,j)\\
    &= K^2\sum_{i,j\in\mathbb N_0} \left(n\cdot \frac{j}{K}\cdot\prod_{k=1}^{n-1}\left(1 - \dfrac{(i+j-1)+j}{K-k}\right)\right)\nu_K(i,j)\\
    &= K^2(1 + \text{o}(1))\sum_{i,j\in\mathbb{N}_0} n\frac{j}{K}\left(1 - \dfrac{(i+j-1)+j}{K}\right)^{n-1}\nu_K(i,j)\\
    &= K^2(1+\text{o}(1))\sum_{k=0}^{n-1}\binom{n-1}{k}(-1)^k \sum_{i,j\in\mathbb N_0} n\frac{j}{K}\dfrac{((i+j-1)+j)^k}{K^k}\nu_K(i,j)\\
    &= (1 + \text{o}(1))\left( nK\sum_{i,j\in\mathbb N_0}j\nu_K(i,j) - \frac{n(n-1)}{2}\sum_{i,j\in\mathbb N_0} (ij + 2j^2 - j)\nu_K(i,j) \right) \\
    &\qquad+ \mathcal{O}\left(\sum_{k=2}^{n-1} n\binom{n-1}{k} \dfrac{(i+j)^{k+1}}{K^{k-1}}\nu_K(i,j)\right)
\end{align*}
The last term vanishes because of the moment bounds \eqref{eq:higher-moments-bound}, where the bounding constant might depend on $L$.
For the first part, recall that 
\[
\sum_{i,j\in\mathbb N_0} j\nu_K(i,j) = \frac{\mathfrak{s}}{K} + \text{o}\left(\frac{1}{K}\right)
\]
and the second sum vanishes in $K$ by \Cref{lem:properties_mixed_moments}, so that
\[
\lim_{K\to \infty} K^2 \sum_{i,j\in\mathbb N_0} p_{i,j}^+(n, K)\nu_K(i,j) = n\mathfrak{s}.
\]

Similarly, we compute for $n\geq 2$ that
\begin{align*}
    &K^2\sum_{i,j\in\mathbb{N}_0}p_{i,j}^-(n, K)\nu_K(i,j)\\
    &= K^2(1 + \text{o}(1))\sum_{i,j\in\mathbb N_0}\binom{n}{2} \dfrac{i(i-1)}{K^2}\left(1 - \frac{(i-1) + (j-1)}{K}\right)^{n-2}\nu_K(i,j)\\
    &= (1 + \text{o}(1))\left(\binom{n}{2}\sum_{i,j\in\mathbb N_0} i(i-1)\nu_K(i,j) + \mathcal{O}\left(\sum_{k=1}^{n-2}\binom{n-2}{k} \dfrac{(i+j)^{k+2}}{K^k}\nu_K(i,j)\right)\right)\\
    &\underset{K\to \infty}{\longrightarrow} \binom{n}{2}(\mathfrak{v}+ \mathfrak{m}).
\end{align*}

For the last estimate, we make use of the moment estimates to conclude that it is asymptotically vanishing.
\end{proof}

\subsection{Convergence}\label{sec:convergence-genealogy}

In order to circumvent the dependence on the population size, we will construct an auxiliary quasi-autonomous Markov process which is asymptotically close to $A_K$ in probability.
In the following, we fix a realisation of $(\xi_K, \hat{\xi}_K, N_K)$.
Recall that $A_K$ can be viewed as a time-inhomogeneous discrete-time Markov chain jumping at points in $\hat{\xi}_K$ with transitions given by \eqref{eq:transition_prob_A_K}.
In the following, we will construct for every $L\in\mathbb{N}$ an auxiliary process $B_K^L$ which is more accessible, but asymptotically close to $A_K$.
Recall the constant $\kappa_L$ from \Cref{lem:7_probabilities_ordered}.
If $N_K\not\in\mathcal{E}_K$ or if $\xi_K$ contains points $(t,i,j)$ with $i+j> \kappa_L K$, we choose $B_K^L$ to jump at points $(T-t, i,j)\in\xi_K$ with transitions
\begin{equation}\label{eq:transitions_of_B_K}
n\mapsto \begin{cases}
    n + 1 & \text{ with probability }p_{i,j}^+(n, N_K^\uparrow)\\
    n - 1 & \text{ with probability  }p_{i,j}^-(n, N_K^\uparrow)
\end{cases}
\end{equation}
independently from the rest.
Assume now that $N_K\in\mathcal{E}_K$ and that every point $(t,i,j)\in\xi_K$ satisfies $i+j \leq \kappa_L K$.
Let $V(t,i,j)$ be an independently drawn uniform variable on $[0,1]$ for every $(T-t,i,j)\in \xi_K$.
Furthermore, define
\[
q_{i,j,K}^\pm(n, N) := \dfrac{p_{i,j}^\pm(n, N_K^\uparrow)}{p_{i,j}^\pm(n, N) + \widehat{p_{i,j}}^{\,\pm}(n, N)},
\]
with $\frac{0}{0} := 0$ and $\frac{c}{0} := +\infty$ whenever $c > 0$.
Next, we define the process $B_K^L$ inductively by setting $B_K^L(0) := A_K(0)$ and through the following jumping rules: the process jumps at most at events $(T-t, i,j)\in\xi_K$.
At every such event, its transitions are as follows:
\begin{enumerate}
    \item Suppose first that $(T-t,i,j)\in\hat\xi_K$ and $A_K(t-) = B_K^L(t-)$.
    If $A_K$ does not jump, then $B_K^L$ also does not jump. 
    If, on the other hand, $A_K$ does jump, then:
    \begin{enumerate}
        \item if $A_K$ jumps up by one, $B_K^L$ jumps also up by one if additionally
        \[
            V(t,i,j) \leq q_{i,j,K}^+(B_K^L(t-), N_K(T-t)).
        \]

        \item if $A_K$ jumps down by one, $B_K^L$ jumps also down by one if additionally
        \[
            V(t,i,j) \leq q_{i,j,K}^-(B_K^L(t-), N_K(T-t))
        \]

        \item otherwise, $B_K^L$ does not jumps.
    \end{enumerate}

    \item If $(T-t,i,j)\in \xi_K\setminus \hat{\xi}_K$ or if $A_K(t-)\neq B_K^L(t-)$, then $B_K^L$ jumps up by one if $V(t,i,j)\leq p_{i,j}^+(B_K^L(t-), N_K^\uparrow)$ and jumps down by one $0 < V(t,i,j) - p_{i,j}^+(B_K^L(t-), N_K^\uparrow)\leq p_{i,j}^-(B_K^L(t-), N_K^\uparrow)$.
\end{enumerate}

The next lemma ensures that $B_K^L$ behaves like a nice autonomous Markov process up to the stopping time $\sigma_K^L := \inf\{t \geq 0\;:\; B_K^L(t)\geq L\}$.

\begin{lemma}\label{lemma:stopped-martingale-problem}
    For every $L\in\mathbb{N}$ and $K$ large enough, the pair $(B_K^L, \sigma_K^L)$ is a solution to the stopped martingale problem for
    \[
    \mathcal{B}_K f(n) = KN_K^\uparrow\sum_{i,j\in\mathbb N_0} \Big( p_{i,j}^+(n, N_K^\uparrow)\big( f(n+1) - f(n)\big) + p_{i,j}^-(n, N_K^\uparrow)\big( f(n-1) - f(n)\big)\Big)\nu_K(i,j).
    \]
\end{lemma}
\begin{proof}
Since $\xi_K$ is a time-homogeneous Poisson point process, its time reversal is too and has the same mean intensity measure.
As such, it is enough to show that, conditionally on $\xi_K$, $B_K^L$ is a Markov chain with transitions \eqref{eq:transitions_of_B_K} at events $(T-t,i,j)\in\xi_K$, at least until $\sigma_K^L$.

First, note that $(A_K, B_K^L)$ conditionally on $(\xi_K, \hat{\xi}_K, N_K)$ is a time-inhomogeneous Markov chain.
Assume that $N_K\in\mathcal{E}_K$ and that every point $(t,i,j)\in\xi_K$ satisfies $i+j\leq \kappa_L K$, as $B_K^L$ has the correct behaviour otherwise.
Consider a point $(T-t,i,j)\in\xi_K$ and suppose that $A_K(t-) = B_K^L(t-)$ for the same reason.
By \Cref{lem:7_probabilities_ordered}, we may conclude that $V_K^\pm(n, N)$ are probabilities for any $1\leq n\leq L$ and $N\in [N_K^\downarrow, N_K^\uparrow]$, so that we may multiply the probabilities and conclude.
\end{proof}

We are now ready to prove \Cref{thm:block-counting-convergence}. To conclude that $A_K$ weakly converges to the process $A$ defined in $\eqref{eq:lineage-counting-lim}$, there are two things left to show: first, that the operator $\mathcal{B}_K$ converges to the generator of $A$, and second, that the stopping time $\tau_K := \inf\{t\geq 0\;:\; A_K(t) \neq B_K(t)\}$ goes to $+\infty$ with $K$. 
We prove these two facts in the following lemmas. 
\Cref{thm:block-counting-convergence} then follows from the fact that the limiting process is non-exploding, so that it suffices to prove the convergence of the processes stopped upon reaching some fixed threshold $L\in\mathbb N$, see e.g.~\cite[Lemma 5.3]{BES04}.

\begin{lemma}
Let $\tau_K := \inf\{t\geq 0\;:\; A_K(t) \neq B_K(t)\}$ $\mathcal{B}_K$ be as in \Cref{lemma:stopped-martingale-problem}.
    For any $f:\mathbb{N}\to\mathbb{R}$ with compact support, we have that $\mathcal{B}_K f\to \mathcal{B}f$ uniformly, where
    \[
    \mathcal{B}f(n) = n\mathfrak{s}\big(f(n+1) - f(n)\big) + \binom{n}{2}(\mathfrak{v} + \mathfrak{m})\big( f(n-1) - f(n)\big),
    \]
    is the generator of \eqref{eq:lineage-counting-lim}.
\end{lemma}
\begin{proof}
This is an immediate consequence of \Cref{lem:7_convergence_of_probabilities}.
\end{proof}

\begin{lemma}
Consider the two stopping times $\tau_K := \inf\{t\geq 0\;:\; A_K(t) \neq B_K(t)\}$ and $\sigma_K^L := \inf\{t \geq 0\;:\; B_K^L(t)\geq L\}$.
    We have that $\tau_K\wedge \sigma_K^L \to \sigma_K^L$ in probability as $K\to \infty$.
\end{lemma}
\begin{proof}
    First, note that we may have $\tau_K \wedge\sigma_K^L < \sigma_K^L$ if $N_K\not\in\mathcal{E}_K$ or if $\xi_K$ contains points $(t,i,j)$ with $i+j> \kappa_L K$.
    Denoting this event by $\mathcal{F}_K$, we have already seen that $\mathbb{P}(\mathcal{F}_K)\to 0$ as $K\to \infty$.
    Note that if we are on $\mathcal{F}_K^c$ and have $\tau_K < \sigma_K^L$, that means that there has been either an event in $\xi_K\setminus \hat{\xi}_K$ before which $A_K = B_K \leq L$ and at which $B_K$ has jumped or there has been an event in $\hat{\xi}_K$ at which $A_K$ has jumped, but $B_K$ has not.
    Now,
    \begin{align*}
        &\mathbb{P}\Big(\mathcal{F}_K^c\cap \{\exists (T-t,i,j)\in \xi_K\setminus\hat\xi_K\text{ s.t. } B_K(t)\neq B_K(t-) = A_K(t-) \leq L\}\Big) \\
        &\leq \sup_{1\leq n\leq L} \mathbb{E}\left[ \mathds{1}_{\mathcal{F}_K^c}\sum_{i,j\in\mathbb N_0} \mathds{1}_{\exists t\in[0,T], (T-t,i,j)\in\xi_K\setminus \hat{\xi}_K}\Big( p_{i,j}^+(n, N_K^\uparrow) + p_{i,j}^-(n, N_K^\uparrow)\Big)\right]\\
        &\leq \sup_{1\leq n\leq L} TKN_K^\uparrow\epsilon_K\sum_{i,j\in\mathbb{N}_0} \Big(p_{i,j}^+(n, N_K^\uparrow) + p_{i,j}^-(n, N_K^\uparrow)\Big) \nu_K(i,j)\\
        &= \mathcal{O}(\epsilon_K) = \text{o}(1).
    \end{align*}
    Secondly,
    \begin{align*}
        &\mathbb{P}\Big(\mathcal{F}_K^c \cap \{\exists (T-t,i,j)\in\hat\xi_K\text{ s.t. } B_K(t-) = A_K(t-)\leq L\text{ and }B_K(t) \neq A_K(t)\}\Big)\\
        &\leq \sup_{\substack{1\leq n\leq L\\ N\in[N_K^\downarrow, N_K^\uparrow]}} \mathbb{E}\left[\sum_{i,j\in\mathbb N_0} \mathds{1}_{\exists (T-t,i,j)\in \hat{\xi}_K}\left( \overline{p}_{i,j}(n,N) + \left( p_{i,j}^+(n, N) - p_{i,j}^+(n, N_K^\uparrow)\right) + \left( p_{i,j}^-(n,N) - p_{i,j}^-(n, N_K^\uparrow)\right)\right)\right] \\
        &\leq \sup_{\substack{1\leq n\leq L\\N\in [N_K^\downarrow, N_K^\uparrow]}} TKN_K\left( \sum_{i,j\in\mathbb N_0}\overline{p}_{i,j}(n, N)\nu_K(i,j) + C_L\epsilon_K\sum_{i,j\in\mathbb N_0} \left(\dfrac{j}{K} + \dfrac{i(i-1)}{K^2}\right)\nu_K(i,j)\right) \\
        &= \text{o}(1) + \mathcal{O}(\epsilon_K)  = \text{o}(1).
    \end{align*}
\end{proof}

\section*{Acknowledgements}
We thank Alison Etheridge for helpful discussions. 
MDP was supported by IRTG 2544 “Stochastic Analysis in Interaction”.
JK has been partially funded by the Deutsche Forschungsgemeinschaft (DFG, German Research Foundation) under Germany’s Excellence Strategy -- The Berlin Mathematics Research Center MATH+ (EXC-2046/1, project ID: 390685689) and through grant CRC 1114 ``Scaling Cascades in Complex Systems", Project Number 235221301, Project C02 ``Interface dynamics: Bridging stochastic and hydrodynamic descriptions".

\printbibliography
\end{document}